\documentclass[10pt,twoside]{siamltex}
\usepackage{amsfonts,epsfig}

\usepackage{graphicx}
\usepackage{amsmath}
\usepackage{amssymb}

\usepackage{color}
\newcommand{\tc}{\textcolor{black}}
\usepackage{eqparbox}
\usepackage[section]{algorithm}
\usepackage{algorithmic}

\newtheorem{remark}{Remark}

\def\e{{\epsilon}}

\def\e{{\varepsilon}}

\newcommand{\N}{\mathbb N}
\newcommand{\R}{\mathbb R}
\newcommand{\C}{\mathbb C}
\newcommand{\Q}{\mathbb Q}

\title{A matrix-algebraic algorithm for the Riemannian logarithm on the {S}tiefel manifold under the canonical metric}
\author{
  Ralf Zimmermann\thanks{Department of Mathematics and Computer Science, University of Southern Denmark (SDU) Odense,
    (zimmermann@imada.sdu.dk).}
}
\begin{document}
\maketitle

\begin{abstract}
  We derive a numerical algorithm for evaluating 
  the Riemannian logarithm on the Stiefel manifold
  with respect to the canonical metric.
  In contrast to the existing  optimization-based approach,
  we work from a purely matrix-algebraic perspective.
  Moreover, we prove that the algorithm converges locally and
  exhibits a linear rate of convergence.
\end{abstract}

\begin{keywords}
  Stiefel manifold, Riemannian logarithm, Riemannian exponential,  Dynkin series, Goldberg series, Baker-Campbell-Hausdorff series
\end{keywords}

\begin{AMS}
  15A16, 
  15B10, 
  15B57, 
  33B30, 
  33F05, 
  53-04, 
  65F60  
\end{AMS}

\section{Introduction}
%
Consider an arbitrary Riemannian manifold $\mathcal{M}$.
Geodesics on $\mathcal{M}$ are locally shortest curves
that are parametrized by the arc length.
Because they satisfy an initial value problem, they are uniquely determined
by specifying a starting point $p_0\in \mathcal{M}$ and a starting velocity
$\Delta \in  T_{p_0}\mathcal{M}$ from the tangent space at $p_0$.
Geodesics give rise to the {\em Riemannian exponential function}
that maps a tangent vector
$\Delta\in T_{p_0}\mathcal{M}$ to the endpoint $\mathcal{C}(1)$ of a geodesic path
$\mathcal{C}:[0,1] \rightarrow \mathcal{M}$ starting at $\mathcal{C}(0) = p_0\in \mathcal{M}$
with velocity $\Delta= \dot{\mathcal{C}}(0)\in T_{p_0}\mathcal{M}$.
It thus depends on the base point $p_0$ and is denoted by  
\begin{equation}
\label{eq:RiemannExp}
   Exp_{p_0}: T_{p_0}\mathcal{M} \rightarrow \mathcal{M}, Exp_{p_0}(\Delta) := \mathcal{C}(1).
\end{equation}
The Riemannian exponential is a local diffeomorphism, \cite[\S 5]{Lee1997riemannian}. 
%
%
%
%
This means that it is locally invertible and that its inverse,
called the {\em Riemannian logarithm} is also differentiable.
Moreover, \tc{the exponential is radially isometric, i.e.,} the Riemannian distance
between the starting point $p_0$ and the endpoint $p_1:= Exp_{p_0}(\Delta)$ on $\mathcal{M}$
is the same as the length of the velocity vector $\Delta$ of the geodesic
$t\mapsto Exp_{p_0}(t\Delta)$ when measured on the tangent space $T_{p_0}\mathcal{M}$, \tc{\cite[Lem. 5.10 \& Cor. 6.11]{Lee1997riemannian}}.
In this way, the exponential mapping gives a local parametrization
from the (flat, Euclidean) tangent space to the (possibly curved) manifold.
This is also referred to as to representing the manifold in {\em normal coordinates}
\cite[\S III.8]{KobayashiNomizu1963}.

The Riemannian exponential and logarithm are important both from the theoretical
perspective as well as in practical applications.
The latter fact holds true in particular,
when $\mathcal{M}$ is a {\em matrix manifold} \cite{AbsilMahonySepulchre2008}.
Examples range from data analysis and signal processing \cite{Gallivan_etal2003, Rahman_etal2005, BolzanoNowakRecht_grouse2010, Rentmeesters2013}
over computer vision \cite{BegelforWerman2006, Lui2012}
to adaptive model reduction and subspace interpolation \cite{BennerGugercinWillcox2015}
and, more generally speaking, optimization techniques on manifolds 
\cite{EdelmanAriasSmith1999,AbsilMahonySepulchre2004, AbsilMahonySepulchre2008}.
This list is far from being exhaustive.
\paragraph{Original contribution}
In the work at hand, we present a matrix-algebraic derivation 
of an algorithm for computing the Riemannian logarithm on the {\em Stiefel manifold}.
The matrix-algebraic perspective allows us to prove local linear convergence.
The approach is based on an iterative inversion of the
closed formula for the associated  Riemannian exponential that
has been derived in \cite[\S 2.4.2]{EdelmanAriasSmith1999}.
Our main tools are Dynkin's explicit Baker-Campbell-Hausdorff formula \cite{rossmann2006lie}
and {G}oldberg's exponential series \cite{Goldberg1956},
both of which represent a solution $Z$ to the matrix equation 
\[
  \exp_m(Z(X,Y)) = \exp_m(X) \exp_m(Y) \left(\Leftrightarrow 
  Z(\log_m(V),\log_M(W)) = \log_m(VW)\right),
\]
where $V=\exp_m(X), W = \exp_m(Y)$ and $\exp_m, \log_m$ are the standard matrix
exponential and matrix logarithm \cite[\S 10, \S 11]{Higham:2008:FM}.
\tc{As an aside}, we improve Thompson's norm bound from \cite{Thompson1989} on $\|Z(X,Y)\|$ for the Goldberg series
by a factor of $2$, where $\|\cdot\|$ is any submultiplicative matrix norm.

The Stiefel log algorithm can be implemented in $\mathcal{O}(10)$ lines of (commented) 
MATLAB \cite{MATLAB:2010} code, which we include in Appendix \ref{app:code}.

\paragraph{Comparison with previous work} To the best of our knowledge, up to now, the only algorithm for evaluating
the Stiefel logarithm appeared in Q. Rentmeesters' thesis \cite[Alg. 4, p.~91]{Rentmeesters2013}.
This algorithm is based on a Riemannian optimization problem.
It turns out that this approach and the ansatz \tc{that is pursued} here, 
though very different in their course of action,
lead to essentially the same numerical scheme.
Rentmeesters observes numerically a linear rate of convergence \cite[p.83, p.100]{Rentmeesters2013}.
Proving linear convergence for \cite[Alg. 4, p.~91]{Rentmeesters2013}
would require estimates on the Hessian, see \cite[\S 5.2.1]{Rentmeesters2013},
\cite[Thm. 4.5.6]{AbsilMahonySepulchre2008}.
In contrast, the derivation presented here uses only elementary matrix algebra
and the \tc{convergence proof given here} formally avoids the requirements of computing/estimating
step sizes, gradients and Hessians that are inherent to \tc{analyzing the convergence of} optimization approaches.
In fact, the convergence proof applies to \cite[Alg. 4, p.~91]{Rentmeesters2013}
and yields the linear convergence of this optimization approach
when using a fixed {\em unit step size}, \tc{but only on a sufficiently small domain.}
The thesis \cite{Rentmeesters2013} was published under a two-years access embargo
and the fundamentals of the work at hand were developed independently before \cite{Rentmeesters2013} was accessible.
\tc{
\paragraph{Transition to the complex case}
The basic geometric concepts of the Stiefel manifold, the  algorithm for the Riemannian log mapping developed here and its convergence proof carry over to complex matrices,
where orthogonal matrices have to be replaced with unitary matrices and skew-symmetric matrices with skew-Hermitian matrices and so forth,
see also \cite[\S 2.1]{EdelmanAriasSmith1999}. The thus adjusted log mapping algorithm was also confirmed numerically to work in the complex case.
}
%
\paragraph{Organization}
Background information on the Stiefel manifold are reviewed in Section \ref{sec:Stiefel_essentials}.
The new derivation for the Stiefel log algorithm is in Section \ref{sec:alg}, 
convergence analysis is performed in
Section \ref{sec:ConvProof},  experimental
results are in Section \ref{sec:experiments}, and the conclusions follow in
Section \ref{sec:conclusions}.
\paragraph{Notational specifics}
\label{sec:Notation}
The $(p\times p)$-{\em identity matrix} is denoted by $I_p\in\R^{p\times p}$.
If the dimension is clear, we will simply write $I$.
The $(p\times p)$-{\em orthogonal group}, i.e., the set of all
square orthogonal matrices is denoted by
\[
  O_{p\times p} = \{\Phi \in \R^{p\times p}| \Phi^T\Phi = \Phi\Phi^T = I_p\}.
\]
The standard matrix exponential and matrix logarithm are denoted by
\[
 \exp_m(X):=\sum_{j=0}^\infty{\frac{X^j}{j!}}, \quad \log_m(I+X):=\sum_{j=1}^\infty{(-1)^{j+1}\frac{X^j}{j}}.
\]
We use the symbols
$Exp^{St}, Log^{St}$ for the Riemannian counterparts on the Stiefel manifold.

When we employ the qr-decomposition of a rectangular matrix $A\in\R^{n\times p}$, 
we implicitly assume that $n\geq p$ and refer to
the `economy size' qr-decomposition $A=QR$, with $Q\in \R^{n\times p}$, $R\in \R^{p\times p}$.
%
%
%

%
\section{The Stiefel manifold in numerical representation}
\label{sec:Stiefel_essentials}
This section reviews the essential aspects of the numerical treatment of
Stiefel manifolds, where we rely heavily on the
excellent references \cite{AbsilMahonySepulchre2008, EdelmanAriasSmith1999}.
The {\em Stiefel manifold} is the compact homogeneous matrix manifold of all column-orthogonal
rectangular matrices
\[
  St(n,p):= \{U \in \R^{n\times p}| \quad U^TU = I_p\}.
\]
The {\em tangent space} $T_USt(n,p)$ at a point $U \in St(n,p)$
can be thought of as the space of velocity vectors of differentiable curves on $St(n,p)$
passing through $U$:
\[
  T_USt(n,p)=\{\dot{\mathcal{C}}(t_o)| \mathcal{C}:(t_0-\epsilon, t_0+\epsilon)\rightarrow St(n,p), \mathcal{C}(t_0)=U\}.
\]
For any matrix representative $U\in St(n,p)$,
the tangent space of $St(n, p)$ at $U$ is represented by
\[
  T_USt(n,p) = \left\{\Delta \in \R^{n\times p}|\quad U^T\Delta = -\Delta^TU\right\}\subset \R^{n\times p}.
\]
Every tangent vector $\Delta \in T_USt(n,p)$  may be written as
\begin{equation}
\label{eq:tang2}
  \Delta = UA + (I-UU^T)T, \quad A \in \R^{p\times p} \mbox{ skew}, \quad T\in\R^{n\times p} \mbox{ arbitrary}.
\end{equation}
%
%
The dimension of both $T_USt(n,p)$ and $St(n,p)$ is $np -\frac{1}{2}p(p+1)$.

Each tangent space carries an inner product 
$\langle \Delta, \tilde{\Delta}\rangle_U = tr\left(\Delta^T(I-\frac{1}{2}UU^T)\tilde{\Delta}\right)$
\tc{with corresponding norm $\| \Delta\|_U = \sqrt{\langle \Delta, \Delta\rangle_U}$.}
This is called the {\em canonical metric} on $T_USt(n,p)$.
It is derived from the 
quotient space representation $St(n,p) = O_{n\times n}/O_{(n-p) \times (n-p)}$
that identifies two square orthogonal matrices in $O_{n\times n}$
as the same point on $St(n,p)$, if their first $p$ columns coincide \cite[\S 2.4]{EdelmanAriasSmith1999}.
Endowing each tangent space with this metric (that varies differentiably in $U$)
turns $St(n,p)$ into a {\em Riemannian manifold}.

We now turn to the Riemannian exponential \eqref{eq:RiemannExp} but for $\mathcal{M} = St(n,p)$.
An efficient algorithm for evaluating the Stiefel exponential
was derived in \cite[\S 2.4.2]{EdelmanAriasSmith1999}.
%
%
%
\tc{The algorithm starts with decomposing} an input tangent vector
$\Delta\in T_USt(n,p)$ into its horizontal and vertical components
with respect to the base point $U$,
\[
   \Delta = UU^T\Delta + (I-UU^T)\Delta \stackrel{(\mbox{qr of } (I-UU^T)\Delta)}{=} UA + Q_ER_E.
\]
Because $\Delta$ is tangent, $A\in\R^{p\times p}$ is skew.
Then the matrix exponential is invoked
to compute 
\begin{equation}
 \label{eq:baby_log}
  \begin{pmatrix}M\\ N_E\end{pmatrix}
          := \exp_m\left(\begin{pmatrix}A & -R_E^T\\ R_E & 0\end{pmatrix}\right)\begin{pmatrix}I_p\\ 0\end{pmatrix}.
\end{equation}
\tc{The final output is\footnote{\tc{The index in $Q_E, R_E, N_E$ is used to emphasize
that these matrices stem from the Stiefel exponential as opposed to the
closely related matrices $Q,R,N$ that will appear in the procedure for the Stiefel logarithm.}}
\begin{equation}
 \label{alg:Stexp}
 \tilde{U} := Exp_{U}^{St}(\Delta) = UM + Q_EN_E \in St(n,p).
\end{equation}
(A MATLAB function for the Stiefel exponential is in the supplement in Appendix \ref{supp:Stexp}.)}
The matrix exponential \tc{in \eqref{eq:baby_log}} is related with the solution of the 
initial value problem that defines a geodesic on $St(n,p)$, see \cite[\S 2.4.2]{EdelmanAriasSmith1999} for details.
It turns out that the main obstacle in computing the inverse of the Stiefel exponential and
thus the Stiefel logarithm is inverting \eqref{eq:baby_log}, i.e.
finding $A,R_E$ given $M, N_E$, compare to \cite[eq. (5.21)]{Rentmeesters2013}.
\section{Derivation of the Stiefel log algorithm}
\label{sec:alg}
Let $U,\tilde{U}\in St(n,p)$ and assume that $\tilde{U}$ is contained in a neighborhood $\mathcal{D}$
of $U$ such that $Exp_U^{St}$ is a diffeomorphism from a neighborhood of $0\in T_USt(n,p)$ onto $\mathcal{D}$.
The central objective is to find $\Delta \in T_USt(n,p)$ such that
$Exp_U^{St}(\Delta) = \tilde{U}.$

Because of Alg. \ref{alg:Stexp}, we know that $\tilde{U}$ allows for a representation
$\tilde{U} = UM + Q_EN_E$.
Hence, we have to determine the unknown matrices $M,N_E\in \R^{p\times p}$, $Q_E\in \R^{n\times p}$,
which feature the following properties: $Q_E^TU=0$ and
$M^TM+N_E^TN_E = I_p$. 
(Note that by \eqref{eq:baby_log},  $M$ and $N_E$ are the left upper and lower $p\times p$ blocks
of a $2p\times 2p$ orthogonal matrix.)
We directly obtain
\[
  M = U^T\tilde{U}, \quad Q_EN_E = (I-UU^T)\tilde{U}.
\]
We compute candidates for $Q_E,N_E$ via
a qr-decomposition
\[
 QN \stackrel{qr}{=}  (I-UU^T)\tilde{U}, \quad Q\in St(n,p).
\]
%
%

\tc{The set of all orthogonal matrices with $M,N$ as an upper diagonal and lower off-diagonal block
is  
parametrized via
\[
    \left\{\begin{pmatrix}M & X\\ N & Y\end{pmatrix}| \quad
    \begin{pmatrix} X\\ Y\end{pmatrix}=\begin{pmatrix} X_0\\ Y_0\end{pmatrix}\Phi, \quad \Phi \in O_{p\times p}
    \right\},
\]
where $(X_0^T, Y_0^T)^T$ is a specific orthogonal completion, computed, say, via the Gram-Schmidt process.}

Thus, the objective is reduced to solving the following
nonlinear matrix equation
\begin{equation}
 \label{eq:fundamentalMatrixEq}
 0 = \begin{pmatrix}0 & I_p\end{pmatrix}
     \log_m\left(
              \begin{pmatrix}M & X_0\\ N & Y_0\end{pmatrix}
              \begin{pmatrix}I_p & 0\\ 0 & \Phi\end{pmatrix}
            \right)
     \begin{pmatrix}0 \\ I_p\end{pmatrix}, \quad \Phi \in O_{p\times p}.
\end{equation}
Writing 
$
\log_m\left(
              \begin{pmatrix}M & X_0\\ N & Y_0\end{pmatrix}
              \begin{pmatrix}I_p & 0\\ 0 & \Phi\end{pmatrix}
            \right)
= \begin{pmatrix} A & -B^T \\ B & C\end{pmatrix}
$,
this means finding a rotation $\Phi$ such that $C=0$.


The first result is that solving \eqref{eq:fundamentalMatrixEq} indeed
leads to the Riemannian logarithm on the Stiefel manifold.
\begin{theorem}
\label{thm:bigthm}
  Let $U,\tilde{U}\in St(n,p)$ and assume that $\tilde{U}$ is contained in a neighborhood $\mathcal{D}$
  of $U$ such that $Exp_U^{St}$ is a diffeomorphism from a neighborhood of $0\in T_USt(n,p)$ onto $\mathcal{D}$.

  Let $M$, $Q_E,N_E$, $Q,N$, $X_0, Y_0$ as introduced in the above setting.
  Suppose that $\Phi \in O_{p\times p}$ solves \eqref{eq:fundamentalMatrixEq}, i.e.,
  \[
   \log_m\left(\begin{pmatrix} M & X_0\Phi\\ N & Y_0\Phi\end{pmatrix}\right) 
    = \begin{pmatrix}A & -B^T \\ B & 0\end{pmatrix}.
  \]
  Define $\Delta := UA + QB \in T_{U}St(n,p)$. 
  Then 
  $Exp_{U}^{St}(\Delta) = \tilde{U}$, i.e.,
  $
    \Delta = Log_{U}^{St}(\tilde{U})$.
\end{theorem}
\begin{proof}
  By construction, it holds $QN = (I-UU^T)\tilde{U}$ and hence
  \begin{equation}
  \label{eq:Qrelations}
    U^TQ = 0, \quad (I-UU^T)Q=Q, \quad QQ^T\tilde{U} = QQ^T (I-UU^T)\tilde{U} = (I-UU^T)\tilde{U}.
  \end{equation}
  Now, we apply the Stiefel exponential Alg. \ref{alg:Stexp} to $\Delta=UA + QB$.
  This gives
  $
    (I-UU^T)\Delta 
      = QB
  $
  and 
  \[
    Q_ER_E \stackrel{qr}{=} QB \Leftrightarrow R_E = \Psi B,
   \mbox{ where }\Psi:= (Q_E^TQ)\in O_{p\times p}. \footnote{The matrices $Q_E$ and $Q$ differ by an orthogonal rotation but span the same subspace.}
  \]
  \tc{With $U^T\Delta = A$, we obtain}
  \begin{subequations}
   \begin{align}
   \nonumber
    \begin{pmatrix}M\\ N_E\end{pmatrix}
       &:= \exp_m\left(\begin{pmatrix}A & -R_E^T\\ R_E & 0\end{pmatrix}\right)
                       \begin{pmatrix}I_p\\ 0\end{pmatrix}\\
   \nonumber
       &= \begin{pmatrix} I & 0 \\ 0 & \Psi\end{pmatrix}
          \exp_m\left(  
            \begin{pmatrix}A & -B^T \\ B & 0\end{pmatrix}\right)
          \begin{pmatrix} I & 0 \\ 0 & \Psi^T\end{pmatrix}
            \begin{pmatrix}I_p\\ 0\end{pmatrix}\\
   \nonumber
       &= \begin{pmatrix} I & 0 \\ 0 & \Psi\end{pmatrix}
          \begin{pmatrix} M & X_0\Phi \\ N & Y_0\Phi \end{pmatrix}
          \begin{pmatrix} I_p\\ 0\end{pmatrix}
        = \begin{pmatrix} M \\ \Psi N\end{pmatrix}.
   \end{align}
  \end{subequations}
 Keeping in mind that $Q_E\Psi = Q_EQ_E^T Q = Q$,
 this leads to an output of
  \begin{equation}
\nonumber 
      Exp_U^{St}(\Delta) =  UM+Q_EN_E = UM + Q_E\Psi N = UM +  Q N = \tilde{U}.
  \end{equation}
 Thus, $\Delta$ is a valid tangent vector in $T_USt(n,p)$
 such that $Exp_U^{St}(\Delta) = \tilde{U}\in St(n,p)$.
 From abstract differential geometry, we know that $Log_U^{St}(\tilde{U}) \in T_USt(n,p)$
 is the unique tangent with $Exp_U^{St}(Log_U^{St}(\tilde{U})) = \tilde{U}$.
 We arrive at the claim
  \[
     \Delta = Log_U^{St}(\tilde{U}).
  \]
\end{proof}
Having established Theorem \ref{thm:bigthm}, we now focus on solving \eqref{eq:fundamentalMatrixEq}.
Let
\begin{subequations}
 \begin{align}
  \nonumber
  V_0 &      := \begin{pmatrix} M & X_0\\ N & Y_0\end{pmatrix}, \quad
  \log_m(V_0):= \begin{pmatrix}A_0 & -B_0^T \\ B_0 & C_0\end{pmatrix}, \\
  \nonumber
  W_0     &  := \begin{pmatrix} I_p & 0\\ 0 & \Phi_0\end{pmatrix},\quad 
  \log_m(W_0) = \begin{pmatrix} 0 & 0 \\ 0 & \log_m(\Phi_0)\end{pmatrix}.
 \end{align}
\end{subequations}
Up to terms of first order, it holds $\log_m(V_0W_0)=\log_m(V_0) + \log_m(W_0)$.
Hence, the choice
\[
  \Phi_0 = \exp_m(-C_0)
\]
gives an approximate solution to \eqref{eq:fundamentalMatrixEq}.
We define 
\begin{equation}
\label{eq:alg_log_iter_idea}
V_1 := \begin{pmatrix} M & X_0\\ N & Y_0\end{pmatrix}
                     \begin{pmatrix} I_p & 0\\ 0 & \Phi_0\end{pmatrix}, \quad  
   \log_m(V_1) := \begin{pmatrix}A_1 & -B_1^T \\ B_1 & C_1\end{pmatrix}
\end{equation}
and iterate. 
This 
\tc{is the essential idea of Alg. \ref{alg:Stlog} for the Riemannian logarithm}.\footnote{This is the same algorithm as 
\cite[Alg. 4, p.~91]{Rentmeesters2013}
that Rentmeesters obtains from his geometrical perspective
when a fixed unit step length is employed and when \cite[\S 5.3]{Rentmeesters2013}
is taken into account.}
In Section \ref{sec:ConvProof} we make use of the Baker-Campbell-Hausdorff
formula \cite[\S 1.3, p. 22]{rossmann2006lie} 
that corrects for the misfit
in the approximative matrix relation $\log_m(VW) \approx \log_m(V) + \log_m(W)$
for two non-commuting matrices $V,W$
in order to show that the above procedure leads to
\[\|C_{k+1}\|_2 < \alpha \|C_k\|_2\]
for all $k\in\N_0$ and a constant $\alpha <1$ and is thus convergent.

Since the Riemannian exponential is a local diffeomorphism, we have to postulate a suitable bound on the distance between 
the input matrices $U$ and $\tilde{U}$.
Suppose that $\|U-\tilde{U}\|_2<\e$. 
Recalling the definitions $M = U^T\tilde{U}$ and $(I-UU^T)\tilde{U} = QN$, this gives the following bounds for the horizontal and the vertical component of $U-\tilde{U}$
with respect to the subspace spanned by $U$:
\[
 \|UU^T(U-\tilde{U})\|_2 = \|I_p - M\|_2 < \e, \hspace{0.2cm} \|(I-UU^T)(U-\tilde{U})\|_2 = \|QN\|_2 = \|N\|_2 < \e.
\]

However, it turns out that for the convergence proof, estimates
on the norms of $X_0$, $Y_0$ and $Y_0-I_p$ are also required.
By the CS-decomposition of orthonormal matrices \cite[Thm 2.6.3, p. 78]{GolubVanLoan}, the diagonal blocks $M$ and $Y_0$
share the same singular values and so do the off-diagonal blocks $N,X_0$.
Hence, $\|N\|_2 = \|X_0\|_2 < \e$.
Let $D_1\Sigma R_1^T$ be the SVD of $M$ and $D_2\Sigma R_2^T$ be the SVD of $Y_0$.
An estimate for the singular values of $M$ can be obtained as follows:
\begin{equation}
\label{eq:NM_connection}
  \e^2 > \|N\|_2^2 = \|N^TN\|_2 = \|I-M^TM\|_2 = \|I-\Sigma^2\|_2 = \max_{\sigma_k} (1- \sigma_k^2),
\end{equation}
where we have used that $ \sigma_1 = \|M\|_2 \leq 1$. 
Now, we replace the $Y_0$ that has been obtained via, say, Gram-Schmidt by
$Y_0R_2D_2^T = D_2\Sigma D_2^T$ (and, correspondingly,  $X_0$ by $X_0R_2D_2^T$).
Essentially, this is the orthogonal Procrustes method, \cite[\S 12.4.1, p.601]{GolubVanLoan},
applied to $\min_{\Psi \in O_{p\times p}} \|I - Y_0\Psi\|_2$.
This operation preserves the orthogonality of $V_0 = \begin{pmatrix} M & X_0\\ N & Y_0\end{pmatrix}$,
but the new $Y_0$ is symmetric with eigenvalue  decomposition
$Y_0 = D_2\Sigma D_2^T$.
This gives 
\[
 \|Y_0-I_p\|_2 = \|\Sigma -I_p\|_2 = \max_{\sigma_k} |1-\sigma_k| < \max_{\sigma_k} (1- \sigma_k^2) < \e^2.
\]
In summary, if $\|U-\tilde{U}\|_2<\e$ and if we start the iterations indicated by \eqref{eq:alg_log_iter_idea} 
with the Procrustes orthogonal completion $X_0, Y_0$ rather than the standard Gram-Schmidt process,
we obtain Alg. \ref{alg:Stlog} with the starting conditions
\begin{equation}
 \label{eq:start_norms}
\|I_p - M\|_2 < \e, \quad \|N\|_2 = \|X_0\|_2 < \e, \quad \|Y_0-I_p\|_2 <\e^2.
\end{equation}
\begin{algorithm}
\caption{Stiefel logarithm, iterative procedure}
\label{alg:Stlog}
\begin{algorithmic}[1]
  \REQUIRE{base point $U\in St(n,p)$ and $\tilde{U}\in St(n,p)$ `close' to base point, $\tau>0$
  convergence threshold}
  \STATE{ $M:= U^T\tilde{U} \in \R^{p\times p}$}
  \STATE{ $QN := \tilde{U} - UM \in \R^{n \times p}$}
    \hfill \COMMENT{(thin) qr-decomp. of normal component of $\tilde{U}$}
  \STATE{ $V_0 := \begin{pmatrix}M&X_0\\N&Y_0\end{pmatrix} \in O_{2p\times 2p}$}
    \hfill \COMMENT{orthogonal completion \tc{and Procrustes}}

  \FOR{ $k=0,1,2,\ldots$}
  \STATE{$\begin{pmatrix}A_k & -B_k^T\\ B_k & C_k\end{pmatrix} := \log_m(V_k)$}
    \hfill \COMMENT{matrix log, $A_k, C_k$ skew}
  \IF{$\|C_k\|_2 \leq \tau$}
    \STATE{break}
  \ENDIF
  \STATE{ $\Phi_{k} = \exp_m{(-C_k)}$}
    \hfill \COMMENT{matrix exp, $\Phi_{k}$ orthogonal}
  \STATE{$V_{k+1} := V_{k}W_k$, where $W_k:=\begin{pmatrix}I_p& 0\\ 0 & \Phi_k\end{pmatrix}$
  }
    \hfill \COMMENT{update}
  \ENDFOR
  \ENSURE{$\Delta := Log_{U}^{St}(\tilde{U}) = U A_k + QB_k \in T_{U}St(n,p)$}
%
\end{algorithmic}
\end{algorithm}
\paragraph{Computational costs}
W.l.o.g. suppose that $n\geq p$. In fact the most important case
in practical applications is $n\gg p$.
Because of the matrix product in step 1 and the qr-decomposition in step 2
of Alg. \ref{alg:Stlog}, the preparatory steps 1--3 require $\mathcal{O}(np^2)$ FLOPS.
The dominating costs in the iterative loop, steps 5--10, are the evaluation
of the matrix logarithm for a $2p$-by-$2p$ orthogonal matrix and the matrix exponential
for a $p$-by-$p$ skew-symmetric matrix in every iteration, both of which can be achieved efficiently
via the Schur decomposition.
The costs are $\mathcal{O}(p^3)$, see \cite[Alg. 7.5.2]{GolubVanLoan}.

A MATLAB function for Alg. \ref{alg:Stlog} is in Appendix \ref{app:code}.
\section{Convergence proof}
\label{sec:ConvProof}
In this section, we establish the convergence of Alg. \ref{alg:Stlog}
under suitable conditions. We state the main result as Theorem \ref{thm:conv_thm}; 
the proof is subdivided into the auxiliary results Lemma \ref{lem:BCHlem}, 
and Lemma \ref{lem:preserve_norms} as well as Lemma \ref{lem:GoldbergSeries}
that appears in Appendix \ref{app:improvedGoldberg}.
An essential requirement is that the point $\tilde{U}\in St(n,p)$ 
that is to be mapped to the tangent space $T_{U}St(n,p)$ is 
sufficiently close to the base point $U\in St(n,p)$ in the sense that
\tc{$\|U - \tilde{U}\|_2 < \e$}.
Throughout, we will make extensive use of Dynkin's explicit BCH formula
\cite[\S 1.3, p. 22]{rossmann2006lie}.
\begin{theorem}
\label{thm:conv_thm}
  Let $U,\tilde{U} \in St(n,p)$. \tc{ Assume that $\|U-\tilde{U}\|_2 < \e$.}
  Let $(V_k)_k$ be the sequence of orthogonal matrices generated by Alg. \ref{alg:Stlog}.

  \tc{If $\e<0.0912$}, 
  then Alg. \ref{alg:Stlog} converges to a limit matrix $V_\infty:= \lim_{k\rightarrow \infty}{V_k}$ such that
  \begin{displaymath}
    \log_m(V_{\infty}):=\begin{pmatrix}A_{\infty} & -B_{\infty}^T \\ B_{\infty} & C_\infty\end{pmatrix}
    = \begin{pmatrix}A_{\infty} & -B_{\infty}^T \\ B_{\infty} & 0\end{pmatrix}.
  \end{displaymath}
  Given a numerical convergence threshold $\tau>0$, see Alg. \ref{alg:Stlog}, line 7,
  the algorithm requires at most
  $k = \lceil \frac{\log(\|C_0\|_2) -\log(\tau)}{\log(2)}\rceil -1$ iteration steps to meet
  the convergence criterion under the above conditions.
\end{theorem}
\begin{remark}
{\em 
 Alg. \ref{alg:Stlog} generates a sequence of orthonormal matrices
  \begin{equation}
  \label{eq:rem1}
     V_{k+1} = V_kW_k = V_0(W_0W_1\dots W_k)
     = V_0 \left(\begin{pmatrix}I_p& 0\\ 0 & \Phi_0\end{pmatrix}\dots 
                 \begin{pmatrix}I_p& 0\\ 0 & \Phi_k\end{pmatrix}\right)
      \in O_{2p\times 2p}.
  \end{equation}
  The proof of Theorem \ref{thm:conv_thm} will show that
  $\lim_{k\rightarrow\infty}W_k=I_{2p}$, see \eqref{eq:phi_limit}.
  Therefore, the sequence of orthogonal products $\Phi_0\dots\Phi_k$
  converges to a limit $\Phi_\infty$ for $k\rightarrow\infty$.
  The limit $\Phi_\infty$ solves \eqref{eq:fundamentalMatrixEq}.
  However, it is not required to actually form $\Phi_\infty$.
}
\end{remark}
In pursuit of the proof of Theorem \ref{thm:conv_thm}, we first show that if the norm of the matrix logarithm of the 
orthogonal matrix $V_k$ produced by Alg. \ref{alg:Stlog} at iteration $k$
is sufficiently small, then the norm of the lower $p$-by-$p$ diagonal block
of the matrix logarithm of the next iterate $V_{k+1}$ is strictly decreasing \tc{by a constant factor}.
\begin{lemma}
\label{lem:BCHlem}
  Let $U,\tilde{U} \in St(n,p)$.
  Let $(V_k)_k\subset O_{2p\times 2p}$ be the sequence of orthogonal matrices generated by Alg. \ref{alg:Stlog}.
  Suppose that at stage $k$, it holds
  \begin{equation}
   \label{eq:norm_cond1}
    \|\log_m(V_k)\|_2 := 
      \|\begin{pmatrix}A_k & -B_k^T \\ B_k & C_k\end{pmatrix}\|_2 < \frac{1}{2}(\sqrt{5}-1).
  \end{equation}
  Then $\log_m(V_{k+1}):=\begin{pmatrix}A_{k+1} & -B_{k+1}^T \\ B_{k+1} & C_{k+1}\end{pmatrix}$ features
  a lower $(p\times p)$-diagonal block of norm
  \begin{displaymath}
    \|C_{k+1}\|_2 < \alpha \|C_k\|_2, \quad 0<\alpha< \frac{1}{2}.
  \end{displaymath}
\end{lemma}
%
%
\begin{proof}
  Given $V_k = \begin{pmatrix}M & X_k \\N & Y_k\end{pmatrix} = 
      \exp_m\left(\begin{pmatrix}A_k & -B_k^T \\ B_k & C_k\end{pmatrix}\right)$,
  Alg. \ref{alg:Stlog} computes the next iterate $V_{k+1}$ via
  \[
    V_{k+1} := V_{k}W_k,
  \]
  where $W_k := \begin{pmatrix}I_p& 0\\ 0 & \exp_m(-C_k)\end{pmatrix}$.
  For brevity, we introduce the notation $L_V:=\log_m(V)$ for the matrix logarithm.
  Recall that $[V,W]:= VW-WV$ denotes the commutator or Lie-bracket of the matrices $V,W$. 
  From Dynkin's formula for the Baker-Campbell-Hausdorff series, see
  \cite[\S 1.3, p. 22]{rossmann2006lie}, we obtain
  \begin{subequations}
    \begin{align}
    \nonumber
      L_{V_{k+1}} &= \log_m(V_{k}W_k)\\
    \nonumber
      & =L_{V_k} + L_{W_k}+ \frac{1}{2}[L_{V_k}, L_{W_k}] \\
    \nonumber
      &\quad+\frac{1}{12}
      \left(
        \bigl[L_{V_k}, [L_{V_k},L_{W_k}]\bigr]
       + \bigl[L_{W_k}, [L_{W_k}, L_{V_k}]\bigr]
      \right)\\
    \nonumber
      & \quad-\frac{1}{24} \biggl[L_{W_k}, \bigl[L_{V_k}, [L_{V_k},L_{W_k}]\bigr] \ \biggr]
       + \sum_{l=5}^{\infty}{z_l(L_{V_k}, L_{W_k})},
    \end{align}
  \end{subequations}
  where $\sum_{l=5}^{\infty}{z_l(L_{V_k}, L_{W_k})} =: h.o.t.(5)$ are 
  the terms of fifth order and higher in the series.
  In the case at hand, it holds
   \begin{subequations}
    \begin{align}
    \nonumber
    L_{V_k} + L_{W_k}& =  \begin{pmatrix}A_k & -B_k^T \\ B_k & C_k\end{pmatrix}
              + \begin{pmatrix} 0 & 0\\ 0 & -C_k\end{pmatrix}
              = \begin{pmatrix}A_k & -B_k^T \\ B_k & 0\end{pmatrix},\\
    \nonumber 
      [L_{V_k}, L_{W_k}]
     & = \begin{pmatrix} 0  & B_k^TC_k \\ C_kB_k & 0\end{pmatrix}.
    \nonumber
    \end{align}
  \end{subequations}
  (Note that the basic idea in designing Alg. \ref{alg:Stlog} was exactly to choose $W_k$ such that the lower
  diagonal block in the BCH-series cancels in the first order terms.)
  
  The third and fourth order terms are
  \begin{subequations}
  \begin{align}
    \nonumber
   &\frac{1}{12} \begin{pmatrix} -2B_k^TC_kB_k  & 
      A_kB_k^TC_k - 2B_k^TC_k^2 \\ 2C_k^2B_k - C_kB_kA_k & B_kB_k^TC_k + C_kB_kB_k^T\end{pmatrix}, \mbox{ and}\\
    \nonumber
   &\frac{1}{24}
    \begin{pmatrix} 0  &  -B_k^TC_k^3 + A_kB_k^TC_k^2 \\ -C_k^3B_k + C_k^2B_kA_k 
                       & B_kB_k^TC_k^2 - C_k^2B_kB_k^T\end{pmatrix}, \mbox{ respectively.}
    \end{align}
  \end{subequations}
  Therefore, the series expansion for the lower diagonal block in $\log_m(V_{k+1})$ 
  starts with the terms of third order:
  \begin{subequations}
  \begin{align}
    \nonumber
      \|C_{k+1}\|_2 &= \|\frac{1}{12}(B_kB_k^TC_k + C_kB_kB_k^T ) - 
                        \frac{1}{24}(B_kB_k^TC_k^2 - C_k^2B_kB_k^T) + h.o.t.(5) \|_2\\
    \label{eq:C_k_est1}
      &\leq \left(\frac{1}{6}\|B_k\|_2^2 + \frac{1}{12}\|B_k\|^2_2\|C_k\|_2
        + \frac{\|h.o.t.(5)\|_2}{\|C_k\|_2}\right)\|C_k\|_2.
    \end{align}
  \end{subequations}
  We tackle the higher order terms via Lemma \ref{lem:GoldbergSeries} from the appendix.
  The lemma applies because $\|C_{k}\|_2 = \|L_{W_k}\|_2 \leq \|L_{V_k}\|_2< \frac{1}{2}(\sqrt{5}-1) < 1$.
  In  this setting, it gives
   \[
     \|h.o.t.(5) \|_2 \leq \sum_{l=5}^{\infty}{\|z_l(L_{V_k}, L_{W_k})\|_2} 
                     < \sum_{l=5}^{\infty}{\|L_{V_k}\|^{l-1}\|L_{W_k}\|_2},
   \]
  since each of the ``letters'' $L_{V_k}, L_{W_k}$ appears at least once
  in every ``word'' that contributes to $z_k(L_{V_k}, L_{W_k})$, see Appendix \ref{app:improvedGoldberg}
  and \cite{Thompson1989, Thompson1989b, VanBruntVisser2015}.

  Writing $s:= \|L_{V_k}\|_2$ and substituting in \eqref{eq:C_k_est1} leads to
  \begin{equation}
    \label{eq:C_k_est2}
     \|C_{k+1}\|_2 < \left(\frac{1}{6}s^2 + \frac{1}{12}s^3 
      +\sum_{l=4}^{\infty}{s^{l}}\right)\|C_k\|_2 =: \alpha \|C_k\|_2.
  \end{equation}
 The proof is complete, if we can show that $\alpha<1$.
 Note that $\sum_{l=4}^{\infty}{s^{l}} = \frac{1}{1-s} -1-s-s^2-s^3$.
 As a consequence
 \[
   \alpha < 1 \Leftrightarrow \frac{s^2}{1-s} - \frac{5}{6}s^2 - \frac{11}{12}s^3 <1.
 \]
 An obvious bound on the size of $s$ is obtained via observing that
 $\frac{s^2}{1-s} < 1$, if $s<\frac{1}{2}(\sqrt{5}-1)\approx 0.618$.
 The corresponding $\alpha$ is $0.4653<\frac{1}{2}$.
 A sharper bound can be obtained via solving the associated quartic equation.
 This shows that the inequality even holds for all $s<0.7111$.
\end{proof}
In order to make use of Lemma \ref{lem:BCHlem},
we establish conditions such that
$\|\log_m(V_k)\|_2  < \frac{1}{2}(\sqrt{5}-1)$ holds
throughout the iterations of Alg. \ref{alg:Stlog}.
%
%
%

%
%

\tc{This is the goal of the the next lemma. It relies on the auxiliary results 
Proposition \ref{prop:log_exp_bound}, Proposition \ref{prop:logm_bound} and Lemma \ref{lem:norm_startV0} from Appendix \ref{app:logm_bound}.
Proposition \ref{prop:log_exp_bound} shows that $\|\exp_m(C) -I\|_2 <\|C\|_2$ for $C$ skew-symmetric;
Proposition \ref{prop:logm_bound} establishes a bound in the opposite direction: if $V$ is orthogonal such that
$\|V-I\|_2 < r$, then $\|\log_m(V)\|_2 < r\sqrt{1-\frac{r^2}{4}}/(1-\frac{r^2}{2})$.
Finally, Lemma \ref{lem:norm_startV0} shows that $\|V_0 - I\|_2 < 2\e$ for the first iterate $V_0$ of Alg. \ref{alg:Stlog},
provided that $\|U-\tilde{U}\|_2 < \e$.
}
%
%
%
\begin{lemma}
\label{lem:preserve_norms}
  Let $U,\tilde{U} \in St(n,p)$ with $\|U-\tilde{U}\|_2 < \e$.
  Let $(V_k)_k\subset O_{2p\times 2p}$ be the sequence of orthogonal matrices generated by Alg. \ref{alg:Stlog},
  where 
  $V_k = \begin{pmatrix}M & X_k \\N & Y_k\end{pmatrix}$.
  Let $\tilde{\e} = 2\e \frac{\sqrt{1-\e^2}}{1-2\e^2}$ and $\hat{\e} := (e^{2\tilde{\e}} - 1) + \e +\e^2$.
  If $0<\e$ is small enough such that $\hat{\e} \frac{\sqrt{1-\frac{\hat{\e}^2}{4}}}{1-\frac{\hat{\e}^2}{2}} < \frac{1}{2}(\sqrt{5}-1)$,
  then
  \begin{displaymath}
    \|\log_m(V_k)\|_2 = 
      \|\begin{pmatrix}A_k & -B_k^T \\ B_k & C_k\end{pmatrix}\|_2 
      < \frac{1}{2}(\sqrt{5}-1) \mbox{ for all } k.
  \end{displaymath}
\end{lemma}
%
%
\begin{proof}
Let $\delta_0 := \frac{1}{2}(\sqrt{5}-1)$.
By Lemma \ref{lem:norm_startV0} from Appendix \ref{app:logm_bound}, it holds 
\[\|\log_m(V_0)\|_2 < 2\e \frac{\sqrt{1-\e^2}}{1-2\e^2} = \tilde{\e} \quad (< \delta_0).\]
%
In particular, $\tilde{\e} > \|\begin{pmatrix}A_0 & -B_0^T \\ B_0 & C_0\end{pmatrix}\|_2 \geq \|C_0\|_2.$
By Alg. \ref{alg:Stlog}, $\Phi_0 = \exp_m(-C_0)$, where $\Phi_0$ is orthogonal. 
By Proposition \ref{prop:log_exp_bound} from  Appendix \ref{app:logm_bound}
\[
 \| \Phi_0 - I\|_2 \leq  \|C_0\|_2 < \tilde{\e}.
\]
Writing $V_1 = I+(V_1 - I)=:I+ E_1$, this leads to the estimate
\begin{subequations}
  \begin{align}
  \nonumber
    \|E_1\|_2 &= \|\begin{pmatrix}M-I & X_0\Phi_0 \\ N & Y_0\Phi_0-I\end{pmatrix}\|_2\\
   \nonumber
     &= \|\begin{pmatrix}M-I & 0 \\ 0 & Y_0(\Phi_0-I)\end{pmatrix} + \begin{pmatrix}0 & 0 \\ 0 & Y_0-I\end{pmatrix} +\begin{pmatrix}0 & X_0\Phi_0 \\ N & 0\end{pmatrix} \|_2\\ 
   \nonumber
    &\leq \max\{\|M-I\|_2, \|Y_0(\Phi_0-I)\|_2\} + \|Y_0-I\|_2 + \max\{\|N\|_2, \|X_0\Phi_0\|_2\}\\
  \nonumber
    &\leq \max\{\e, \|Y_0\|_2\|(\Phi_0-I)\|_2\} + \e^2 + \e\leq \tilde{\e} + \e^2 + \e\\
  \nonumber
    &\leq (e^{2\tilde{\e}} - 1) + \e +\e^2 = \hat{\e},
  \end{align}
\end{subequations}
where we have used \eqref{eq:start_norms} and the fact that $\|Y_0\|_2\leq 1$, see \eqref{eq:num_rad1}, \eqref{eq:num_rad2}. 
By Lemma \ref{lem:norm_startV0}, $\|\log_m(V_1)\|_2 < \hat{\e} \sqrt{1-\frac{\hat{\e}^2}{4}}/(1-\frac{\hat{\e}^2}{2})< \delta_0$.
Thus, the claim holds for $k=0,1$.

Lemma \ref{lem:BCHlem} applies to $\|\log_m(V_0)\|_2$ and leads to $\|C_1\|_2 < \frac{1}{2}  \|C_0\|_2 < \frac{\tilde{\e}}{2}$
for the lower diagonal block $C_1$ of the next iterate $\log_m(V_1)$.
Therefore, by using Proposition \ref{prop:log_exp_bound} once more, we see that
\[
  \|\Phi_1-I\|_2 \leq \|C_1\|_2 < \frac{\tilde{\e}}{2}.
\]
By induction, we obtain $V_k = I+(V_k - I)=:I+ E_k$ with
\begin{subequations}
  \begin{align}
  \nonumber
    \|E_k\|_2 &= \|\begin{pmatrix}M & X_0\hat{\Phi}_{k-1} \\ N & Y_0\hat{\Phi}_{k-1}\end{pmatrix} - I\|_2\\
   \nonumber
     &= \|\begin{pmatrix}M-I & 0 \\ 0 & Y_0(\hat{\Phi}_{k-1}-I)\end{pmatrix} + \begin{pmatrix}0 & 0 \\ 0 & Y_0-I\end{pmatrix} +\begin{pmatrix}0 & X_0\hat{\Phi}_{k-1} \\ N & 0\end{pmatrix} \|_2\\ 
   \nonumber
    &\leq \max\{\|M-I\|_2, \|Y_0(\hat{\Phi}_{k-1}-I)\|_2\} + \|Y_0-I\|_2 + \max\{\|N\|_2, \|X_0\hat{\Phi}_{k-1}\|_2\}\\
  \label{eq:norm_Vk_I}
    &\leq \max\{\e, \|Y_0\|_2\|\hat{\Phi}_{k-1}-I\|_2\} + \e^2 + \e.
  \end{align}
\end{subequations}
where $\hat{\Phi}_{k-1} = \Phi_0\cdots \Phi_{k-1}$.

We can estimate $\|\hat{\Phi}_{k-1}-I\|_2$ as follows:
By the induction hypothesis, we assume that we have checked that $\|\log_m(V_{j})\|_2 < \delta_0$ for $j=0,\ldots,k-1$.
Hence, Lemma \ref{lem:BCHlem} ensures that $\|C_{j}\|_2 < \frac{1}{2}  \|C_{j-1}\|_2 <\ldots < \frac{1}{2^{j}}\|C_0\|_2  < \frac{\tilde{\e}}{2^{j}}$
for the lower diagonal block of $\log_m(V_{j})$, $j=0,\ldots,k-1$.
As above, this gives $\|\Phi_j - I\|_2 \leq \|C_{j}\|_2 <\frac{\tilde{\e}}{2^{j}}$.
We thus may write $\Phi_j = I + (\Phi_j- I)=: I + \Gamma_j$ with $\|\Gamma_j\|_2 =: g_j < \frac{\tilde{\e}}{2^{j}}$.
This gives
\begin{equation}
\label{eq:aux_Lemma5}
\|\hat{\Phi}_{k-1}-I\|_2 = \|(I + \Gamma_1)\cdots(I + \Gamma_{k-1})-I\|_2
 \leq (1+g_1)\cdots(1+g_{k-1}) - 1.
\end{equation}
It holds
\[
 \ln\left(\prod_{j=0}^{k-1}{(1+g_j)}\right) = \sum_{j=0}^{k-1}{\ln(1+g_j)} \leq \sum_{j=0}^{k-1}{g_j}\leq \sum_{j=0}^{\infty}{\frac{\tilde{\e}}{2^{j}}} = 2\tilde{\e}.
\]
Using this estimate in \eqref{eq:aux_Lemma5} gives
\[
  \|\hat{\Phi}_{k-1}-I\|_2 < e^{2\tilde{\e}} -1
\]
and we finally arrive at
\[
  \|E_k\|_2 \leq (e^{2\tilde{\e}} -1) + \e^2 + \e = \hat{\e}.
\]
Recalling \eqref{eq:norm_Vk_I}, we have $V_k = I+E_k$ with $\|E_k\|_2< \hat{\e}$ at every iteration $k$.
By  Lemma \ref{lem:norm_startV0}, $\|\log_m(V_k)\|_2<  \hat{\e} \frac{\sqrt{1-\frac{\hat{\e}^2}{4}}}{1-\frac{\hat{\e}^2}{2}}$
and we see that the postulate on the size of $\e$ is such that $\|\log_m(V_k)\|_2 < \delta_0$.
Thus Lemma \ref{lem:BCHlem} indeed applies at iteration $k$, which closes the induction.
\end{proof}
\tc{{\bf Remark:} 
The inequality $\hat{\e} \frac{\sqrt{1-\frac{\hat{\e}^2}{4}}}{1-\frac{\hat{\e}^2}{2}} < \delta_0$ holds precisely
for $\hat{\e} < \sqrt{2}\left(1-\frac{1}{\sqrt{1+\delta_0^2}}\right)^{\frac{1}{2}}=:\hat{\e}_0$.
A further calculations shows that if $\e< 0.0912$, then 
$\hat{\e}=(e^{2\tilde{\e}} -1) + \e^2 + \e<\hat{\e}_0$, i.e., the conditions of Lemma \ref{lem:preserve_norms} hold,
for all $\e< 0.0912$.}

With the tools established above at hand, we are now in a position to prove Theorem \ref{thm:conv_thm}.
\begin{proof}[Theorem \ref{thm:conv_thm}]
 Let $(V_k)_{k\in\N_0}$ be the sequence of orthogonal matrices generated by Alg. \ref{alg:Stlog}.
 By Lemma \ref{lem:BCHlem} and Lemma \ref{lem:preserve_norms}, it holds 
 \begin{equation}
  \label{eq:conv_C_k}
  \|\log_m{V_k}\|_2 :=\|\begin{pmatrix}A_k & -B_k^T \\ B_k & C_k\end{pmatrix}\|_2<\frac{1}{2}(\sqrt{5}-1), \quad
    \|C_{k+1}\|_2  < \alpha^{k+1} \|C_{0}\|_2
\end{equation}
  for all $k \geq 0$, where $ 0<\alpha <\frac{1}{2}$.
 From this equation and the continuity of the matrix exponential, we obtain
 \begin{equation}
 \label{eq:phi_limit}
  \lim_{k\rightarrow\infty} W_k = \lim_{k\rightarrow\infty}\begin{pmatrix} I_p & 0 \\ 0 & \exp_m(-C_k)\end{pmatrix}
  = \begin{pmatrix} I_p & 0 \\ 0 & I_p\end{pmatrix}.
 \end{equation}
 The convergence result is now an immediate consequence of Alg. \ref{alg:Stlog}, step 10.
 The upper bound on the iteration count required for numerical convergence
 is also obvious from \eqref{eq:conv_C_k}.
\end{proof}
\section{Examples and experimental results}
\label{sec:experiments}
In this section, we discuss a special case that can be
treated analytically. Following, we present numerical results
on the performance of Alg. \ref{alg:Stlog}.
\subsection{A special case}
Here, we consider the special situation, where the two points 
$U,\tilde{U} \in St(n,p)$ are such that their columns span the same 
subspace.\footnote{We may alternatively express this by saying that $[U]:=\mbox{colspan}(U)$
and $[\tilde{U}]:=\mbox{colspan}(\tilde{U})$ are the same points on the Grassmann manifold
$[U] = [\tilde{U}] \in Gr(n,p)$.}
Hence, there exists an orthogonal matrix $M\in O_{p\times p}$ such that
$\tilde{U} = UM = UM + (I-UU^T)0$.
In this case, Alg. \ref{alg:Stlog} produces the initial matrices
$V_0  = \begin{pmatrix}M & 0 \\ 0 & Y_0\end{pmatrix}$
and $\Phi_0 = \exp_m(-\log_m(Y_0)) = Y_0^{-1}$.
Note that the corresponding 
$W_0=\begin{pmatrix}I_p & 0 \\ 0 & Y_0^{-1}\end{pmatrix}$ commutes with $V_0$.
Thus, we have the reduced BCH formula 
$\log_m(V_0W_0) = \log_m(V_0) + \log_m(W_0) = \begin{pmatrix}\log_m(M) & 0 \\ 0 & 0\end{pmatrix}$,
i.e., Alg. \ref{alg:Stlog} converges after a single iteration
and gives 
\begin{equation}
\label{eq:specialLog}
  Log_U^{St}(UM)= U \log_m(M). 
\end{equation}
(Of course, it is also straight forward to show this directly without invoking Alg. \ref{alg:Stlog}.)
Let $\sigma(M) = \{e^{i\varphi_1},\dots, e^{i\varphi_p}\}$ be the spectrum of $M\in O_{p\times p}$ and
suppose that $M$ is such that none of its eigenvalues is on the negative real axis,
i.e., $\varphi_j \in (-\pi, \pi)$.
Then, the maximal Riemannian distance between two points $U$ and $UM$ is bounded by
\begin{subequations}
 \begin{align}
  \nonumber
     dist(U,UM) &= \sqrt{\langle U\log_m(M), U\log_m(M) \rangle_U}\\
  \nonumber
    & = \left(
        \frac{1}{2} tr(\log_m(M)^T\log_m(M))\right)^{\frac{1}{2}}  
      = \biggl(\frac{1}{2} \sum_{j=1}^p{\varphi_j^2}\biggr)^{\frac{1}{2}}.
 \end{align}
\end{subequations}
As a consequence 
\[ 
  dist(U,UM) < 
        \left\{
        \begin{array}{ll}
         \sqrt{\frac{p}{2}} \pi, & \quad  p \mbox{ even},\smallskip\\
         \sqrt{\frac{p-1}{2}} \pi, & \quad  p \mbox{ odd}.
        \end{array}
         \right.
\]
The latter fact holds, because the eigenvalues of $M$ come in complex conjugate pairs.
Hence, if $p$ is odd, there is at least one real eigenvalue $\lambda_j = e^{i\varphi_j}$
and because $\varphi_j \in (-\pi, \pi)$, there is at least one zero argument $\varphi_j = 0$.
Related is \cite[eq. (2.15)]{EdelmanAriasSmith1999}.
%
\subsection{Numerical performance}
\label{sec:experiments1}
First, we try to mimic the experiments featured in \cite[\S 5.4]{Rentmeesters2013}.
Fig. 5.5 (lower left) of the aforementioned reference shows the average iteration count
when applying the optimization-based Stiefel logarithm to matrices within a Riemannian
annulus of inner radius $0.4\pi$ and outer radius $0.44\pi$ around $(I_p, 0)^T\in St(n,p)$
for dimensions of $n=10, p=2$.
Convergence is detected, if $\|C_k\|_F < \tau=10^{-7}$, where $C_k$ is the same as in Alg. \ref{alg:Stlog}.
(\cite[Alg. 4, p.~91]{Rentmeesters2013} uses $\tau^2 <10^{-14}$).
Since \cite[\S 5.4]{Rentmeesters2013} does not list the precise input data,
we create comparable data randomly.
To this end, we fix an arbitrary point $U\in St(10,2)$
and create artificially but randomly another point $\tilde{U}\in St(10,2)$
such that the Riemannian distance from $U$ to $\tilde{U}$ is exactly $0.44\pi$.
For full comparability, we replace the $2$-norm in Alg. \ref{alg:Stlog}, line 7 with the Frobenius norm.
We average over $1000$ random experiments and arrive at an average iteration count of $\bar{k}=7.83$.
A MATLAB script that performs the required computations is available in Appendix \ref{supp:ex52}.
When the distance of $U$ and $\tilde{U}$ is lowered to $0.4\pi$, the average iteration count drops
to a value of $\bar{k}=6.92$.

As a second experiment, we now return to the $2$-norm and 
lower the convergence threshold to $\|C_k\|_2 < \tau = 10^{-13}$
in the convergence criterion of Alg. \ref{alg:Stlog}.
We create randomly points $U, \tilde{U} \in St(n,p)$ that are also 
a Riemannian distance of $0.44\pi$ away from each other,
where we consider various different dimensions $(n,p)$, see Table \ref{tab:convhist}.
We apply Alg. \ref{alg:Stlog} to compute $\Delta = Log_U^{St}(\tilde{U})$.
\begin{table}[htbp]
\tc{
  \caption{Convergence of Alg. \ref{alg:Stlog} for random data to an accuracy of $\|C_k\|_2\leq 10^{-13}$.}
  \label{tab:convhist}
  \centering
  \begin{tabular}{|c|c|c|c|c|c|} \hline
   $(n,p)$          & dist$\left(U,\tilde{U}\right)$ & $\|U-\tilde{U}\|_2$ &iters. &  $\|\Delta - Log^{St}_U(\tilde{U})\|_2$ & time \\ \hline
    (10,2)          & $0.44\pi$                      & 1.0179              & 16        &  8.7903e-15                         & 0.01s            \\
    (10,2)          & $0.89\pi$                      & 1.7117              & 95        &  4.1934e-13                         & 0.06s            \\
    (1,000, 200)    & $0.44\pi$                      & 0.1616              & 5         &  1.5119e-14                         & 0.7s             \\
    (1,000, 200)    & $0.89\pi$                      & 0.3256              & 7         &  1.7272e-14                         & 0.8s             \\
    (1,000, 900)    & $0.44\pi$                      & 0.1234              & 4         &  9.6999e-14                         & 16.1s            \\
    (1,000, 900)    & $0.89\pi$                      & 0.2491              & 5         &  7.9052e-14                         & 21.0s            \\
    (100,000, 500)  & $0.44\pi$                      & 0.0875              & 4         &  5.9857e-14                         & 13.1s            \\
    (100,000, 500)  & $0.89\pi$                      & 0.1768              & 5         &  6.1041e-14                         & 14.0s            \\ \hline
  \end{tabular}
}
\end{table}
%
%
%
%
%
%
%
%
\begin{figure}[htbp]
  \centering
  \includegraphics[width=1.0\textwidth]{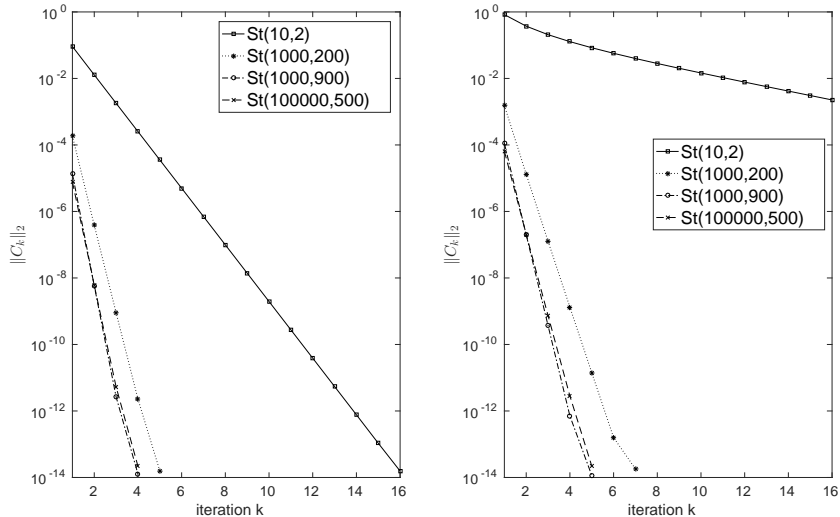}
  \caption{Convergence of Alg. \ref{alg:Stlog} for random data $U,\tilde{U} \in St(n,p)$
           for various $n$ and $p$. 
           Convergence accuracy is set to $\|C_k\|_2\leq 10^{-13}$.
           Left: convergence graphs for dist$(U,\tilde{U})=0.44\pi$; 
           right: for dist$(U,\tilde{U})=0.89\pi$.}
             \label{fig:convplots}
\end{figure}
Fig. \ref{fig:convplots} shows the associated convergence histories.
The associated computation times\footnote{as measured on a Dell desktop computer endowed with 
six processors of type Intel(R) Core(TM) i7-3770 CPU@3.40GHz} are listed
in Table \ref{tab:convhist}.
As can be seen from the figure and the table, Alg. \ref{alg:Stlog} converges slowest (in terms of the iteration count)
in the case of $St(10,2)$. Note that in this case, the constant \tc{$\|U-\tilde{U}\|_2$} that played a major
role in the convergence analysis of Alg. \ref{alg:Stlog} is largest.
Moreover, we observe that the algorithm converges in all test cases even though in only one of the experiments
the theoretical convergence guarantee $\|U_0-\tilde{U}\|_2<0.09$ is satisfied, so that the theoretical bound derived here 
can probably be improved.
Table \ref{tab:convhist} suggests that the impact of the size of \tc{$\|U-\tilde{U}\|_2$}
on the iteration count is \tc{more direct} than that of the actual Riemannian distance.
We repeat the exercise with random data $U, \tilde{U} \in St(n,p)$ that are 
a distance of $0.89\pi$ apart, which is the lower bound for the injectivity 
radius on the Stiefel manifold given in \cite[eq. (5.14)]{Rentmeesters2013}.
In the case of $St(10,2)$, we hit a random matrix pair $U, \tilde{U}$,
where the associated value\tc{$\|U-\tilde{U}\|_2$} is so large 
that the conditions of Theorem \ref{thm:conv_thm}  {\em and} Lemma \ref{lem:BCHlem}, Lemma \ref{lem:preserve_norms} do not hold.
In fact, we have $\|\log_m(V_0)\|_2 = 3.141$ for the starting point of Alg. \ref{alg:Stlog}
in this case, \tc{which is close to $\pi$}. Yet, the algorithm converges, but very slowly so, see Table \ref{tab:convhist}, second row
and Fig. \ref{fig:convplots}, right side.
In all of the other cases,
Alg. \ref{alg:Stlog} converges in well under ten iterations, even for the larger test cases.

A MATLAB script that performs the required computations is available in Appendix \ref{supp:ex52}.
\subsection{Dependence of the convergence on the Riemannian and the Euclidean distance}
\label{sec:experiments2}
In this section, we examine the convergence of Alg. \ref{alg:Stlog} depending on the Riemannian
distance $dist(U,\tilde{U})$ and the distance $\|U - \tilde{U}\|_2$ in the Euclidean operator-$2$-norm.
To this end, we create a random point $U\in St(n,p)$ with MATLAB by computing the thin qr-decomposition
of an $(n\times p)$ matrix with entries sampled uniformly from $(0,1)$.
Likewise, we create a random tangent vector $\Delta\in T_USt(n,p)$ by chosing randomly a skew-symmetric matrix $A= \tilde{A}-\tilde{A}^T\in \R^{p\times p}$
and a matrix $T\in \R^{n\times p}$, where the entries of $\tilde{A}$ and $T$ are again uniformly sampled from $(0,1)$,
and setting $\tilde{\Delta} = UA + (I-UU^T)T$. We normalize $\tilde{\Delta}$ according to the canonical metric 
$\Delta = \frac{\tilde{\Delta}}{\sqrt{\langle \tilde{\Delta}, \tilde{\Delta}\rangle_U}}$, see Section \ref{sec:Stiefel_essentials}.
In this way, we obtain for every $t\in [0,\pi)$ a point $\tilde{U}=U(t)$ that is a Riemannian distance of $dist(U,U(t)))= \| t\Delta\|_U = t$ away from $U$.

We discretize the interval $[0.1,0.9\pi)$ by $100$ equidistant points $\{x_k| k=1,\ldots,100\}$ and compute
\begin{itemize}
 \item the number of iterations until convergence when computing $\log^{St}_U(U(t_k))$ with Alg. \ref{alg:Stlog} for $k=1,\ldots,100$.
 \item the distance in spectral norm $\|U-U(t_k)\|_2$, $k=1,\ldots,100$.
 \item the norm of the matrix logarithm of the first iterate $\|\log_m(V_0)\|_2$ from Alg. \ref{alg:Stlog}, step 3.
\end{itemize}

The results are displayed in Figures \ref{fig:Plot_dist_vs_conv_St10000_400} -- \ref{fig:Plot_dist_vs_conv_St4_2}
for dimensions of $St(10,000, 400)$, $St(100,10)$ and $St(4,2)$, respectively. In all cases, the convergence threshold was set to $\|C_l\|_2<\tau = 10^{-13}$.
The algorithm converged in all cases, where $\|\log_m(V_0)\|_2<\pi$ and produced a tangent vector $\Delta(t_k) := \log^{St}_U(U(t_k))$
of accuracy $\|\Delta(t_k) - t_k\Delta\|_2<10^{-13}$.
A MATLAB script that performs the required computations is available in Appendix \ref{supp:ex53}.
\begin{figure}[htbp]
  \centering
  \includegraphics[width=1.0\textwidth]{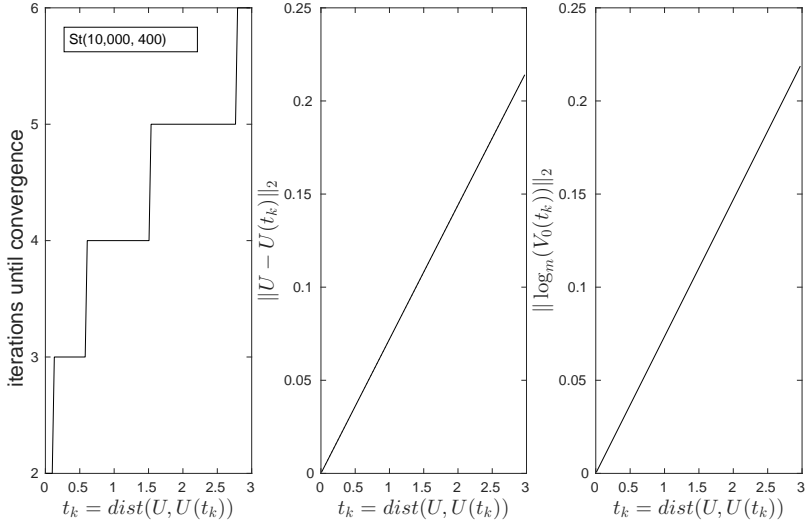}
  \caption{Convergence of Alg. \ref{alg:Stlog} for $U,\tilde{U}=U(t_k) = Exp^{St}_U(t_k\Delta) \in St(n,p)$,
	    where $\Delta$ is a random tangent vector of canonical norm $1$ and $n=10,000$, $p=400$. 
           Convergence accuracy is set to $\|C_k\|_2\leq 10^{-13}$.
           Left: number of iterations until convergence vs. $dist(U,\tilde{U})$;
           middle: $\|U - \tilde{U}\|_2$ vs. $dist(U,\tilde{U})$;
           right:  $\|\log_m(V_0)\|_2$ vs. $dist(U,\tilde{U})$.}
             \label{fig:Plot_dist_vs_conv_St10000_400}
\end{figure}
In the case of $St(4,2)$, the algorithm starts to fail for $t_k\approx\frac{\pi}{2}$, where $\|\log_m(V_0)\|_2$ jumps to a value
of $\pi$. This indicates that $V_0$ features (up to numerical errors) an eigenvalue $\lambda = -1$ so that the standard principal
matrix logarithm is no longer well-defined. In all the experiments that were conducted, this behavior was observed only
for small values of $p<8$, while there was never produced a random data set where Alg. \ref{alg:Stlog} failed for $t<0.9\pi$ and $p>10$.
\begin{figure}[htbp]
  \centering

  \includegraphics[width=1.0\textwidth]{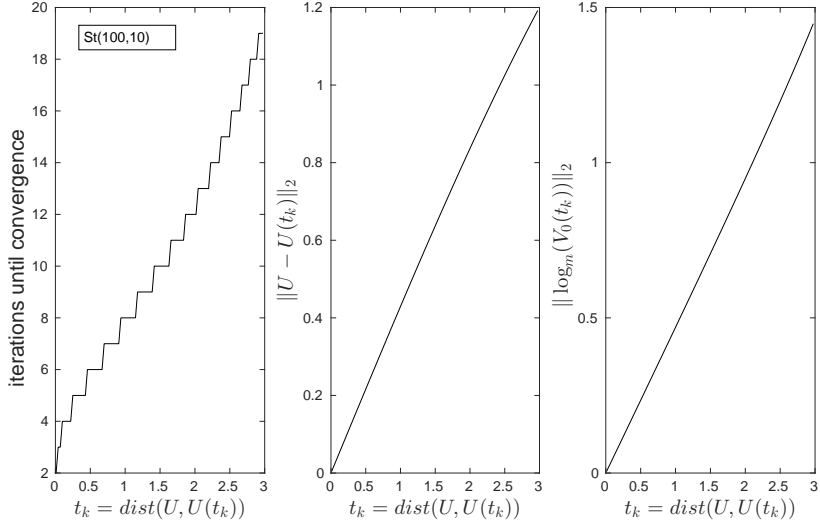}
  \caption{Same as Fig. \ref{fig:Plot_dist_vs_conv_St10000_400}, but for $n=100$, $p=10$.}
    \label{fig:Plot_dist_vs_conv_St100_10}
\end{figure}
\begin{figure}[htbp]
  \centering
  \includegraphics[width=1.0\textwidth]{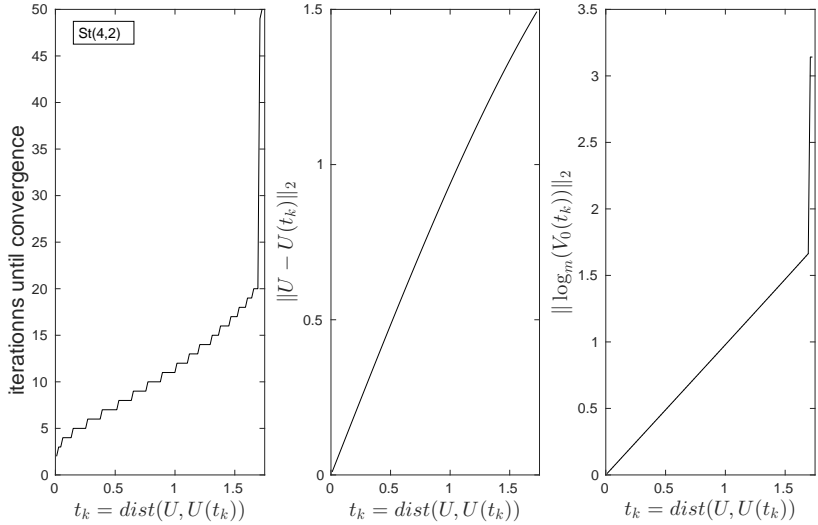}
  \caption{Same as Fig. \ref{fig:Plot_dist_vs_conv_St10000_400}, but for $n=4$, $p=2$.}
    \label{fig:Plot_dist_vs_conv_St4_2}
\end{figure}
The figures suggest that for small column-numbers $p$, the ratio between the Riemannian distance $dist(U,\tilde{U})$ and 
the spectral distance $\|U-\tilde{U}\|_2$ is smaller than in higher dimensions.
Moreover, for smaller $p$, it seems to be more likely to hit a random tangent direction along which Alg. \ref{alg:Stlog}
fails early than for higher $p$. This may partly be explained by the star-shaped nature of the domain of injectivity of
the Riemannian exponential, \cite[Lemma 5.7]{Lee1997riemannian}, and the richer variety of directions in higher dimensions.

From these observations, it is tempting to conjecture that Alg. \ref{alg:Stlog} will converge, whenever $\|\log_m(V_0)\|_2<\pi$.
However, these results are based on a limited notion of randomness and a more thorough examination
of the numerical behavior of Alg. \ref{alg:Stlog} is required to obtained conclusive results,
which is beyond the scope of this work.
Note that the domain of convergence of Alg. \ref{alg:Stlog} is related to the injectivity radius of $St(n,p)$
but it does not have to be the same.
In Appendix \ref{supp:critical_convergence} from the supplement, we state an explicit example in $St(4,2)$, where Alg. \ref{alg:Stlog}
produces a first iterate $V_0$ with $\lambda = -1$ for an input pair $U, \tilde{U}\in St(4,2)$ with $dist(U,\tilde{U}) = \frac{\pi}{2}$,
while the injectivity radius is estimated to be $\approx 0.71\pi$ in \cite[\S 5]{Rentmeesters2013}.
An analytical investigation in $St(4,2)$ might be possible and may shed more light on the precise value of the Stiefel manifold's injectivity
radius.
%
%
\section{Conclusions and outlook}
\label{sec:conclusions}
We have presented a matrix-algebraic derivation of an algorithm for evaluating
the Riemannian logarithm $Log^{St}_U(\tilde{U})$ on the Stiefel manifold.
In contrast to \cite[Alg. 4, p.~91]{Rentmeesters2013},
the construction here is not based on an optimization procedure but on an iterative
solution to a non-linear matrix equation.
Yet, it turns out that both approaches lead to essentially the same numerical scheme.
More precisely, our Alg. \ref{alg:Stlog} coincides with \cite[Alg. 4, p.~91]{Rentmeesters2013},
when a unit step size is employed in the optimization scheme associated
with the latter method.
Apart from its comparatively simplicity, 
a key benefit is that our matrix-algebraic approach
allows for a convergence analysis that does not require estimates on gradients nor Hessians
and we are able to prove that the convergence rate of Alg. \ref{alg:Stlog}
is at least linear.
This, in turn, proves the local linear convergence of \cite[Alg. 4, p.~91]{Rentmeesters2013}
when using a unit step size.
The algorithm shows a very promising performance in numerical experiments,
even when the dimensions $n,p$ become large.

So far, we have carried out a theoretical {\em local} convergence analysis.
Open questions to be tackled in the future include estimates
on how large the convergence domain of Alg. \ref{alg:Stlog} is in terms of
the Riemannian distance of the input points $dist(U,\tilde{U})$.
This is related with the question of determining the injectivity radius
of the Stiefel manifold. Estimates on the injectivity radius are featured
in \cite[\S 5.2.1]{Rentmeesters2013}.

\appendix
\section{A sharper majorizing series for Goldberg's Exponential series}
\label{app:improvedGoldberg}
As an alternative to Dynkin's BCH formula of nested commutators, 
Goldberg has shown in \cite{Goldberg1956} that the solution
to the exponential equation
\[
 \exp_m(X)\exp_m(Y) = \exp_m(Z)
\]
can be written as a formal series
\begin{equation}
\label{eq:Goldberg}
  Z =  X + Y + \sum_{k=2}^{\infty}{z_k(X,Y)}, \quad z_k(X,Y) = \sum_{w, |w|=k}{g_w w}.
\end{equation}
Each term $z_k(X,Y)$ in \eqref{eq:Goldberg} is the sum over all {\em words} of length $k$ in the alphabet $\{X,Y\}$.
For example, $YXYX^2$ and $X^2YXY^2$ are such words of length $5$ and $6$ and thus contributing 
to $z_5(X,Y)$ and $z_6(X,Y)$, respectively.
The coefficients are rational numbers $g_w\in \Q$, called Goldberg coefficients.

Thompson \cite{Thompson1989} has shown that the series converges provided
that $\|X\|,\|Y\| \leq \mu < 1$ for any submultiplicative norm $\|\cdot\|$.
More precisely, his result is that $\|z_k(X,Y)\| = \|\sum_{w, |w|=k}{g_w w}\| \leq 2\mu^k$
for $k\geq 2$, see also \cite[eq. 2]{Thompson1989b}.
In the next lemma, we improve this bound by cutting the factor $2$.
\begin{lemma}
 \label{lem:GoldbergSeries}
 Let $\|X\|,\|Y\| \leq \mu < 1$.
 The Goldberg series is majorized by
  \[
   \|Z\| < \|X\| + \|Y\| + \sum_{k=2}^{\infty}{\mu^k}.
  \]
\end{lemma}
\begin{proof}
 One ingredient of Thompson's proof is the following basic estimate
 on binomial terms:
  \begin{equation}
  \label{eq:thompson_aux}
    m \begin{pmatrix}m-1 \\ \lfloor \frac{m}{2} \rfloor \end{pmatrix} \geq 2^{m-1}.
  \end{equation}
 Here, $\lfloor x \rfloor$ denotes the largest integer smaller or equal to $x$.
 Thompson's argument is that 
 $2^{m-1} = (1+1)^{m-1} = \sum_{l=0}^{m-1}{\begin{pmatrix}m-1 \\  l \end{pmatrix}}$
 and that $\begin{pmatrix}m-1 \\ \lfloor \frac{m}{2} \rfloor \end{pmatrix}$ is the largest 
 out of the $m$ terms in the binomial sum. (It appears twice, if $m-1$ is odd.)
 In the following, we prefer to write this term with using the ceil-operator as 
  $\begin{pmatrix}m-1 \\ \lfloor \frac{m}{2} \rfloor \end{pmatrix}
   = \begin{pmatrix}m-1 \\ \lceil \frac{m-1}{2} \rceil \end{pmatrix}$,
 because in this way, the same index $m-1$ appears in the upper and lower entry
 of the binomial coefficient.

 For larger $m$, the inequality \eqref{eq:thompson_aux} can in fact be improved by a factor of $2$:
  \begin{equation}
    \label{eq:thompson_aux2}
  \mbox{Claim: }  m \begin{pmatrix}m-1 \\ \lceil \frac{m-1}{2} \rceil  \end{pmatrix} > 2^{m} \mbox{ for all } m\geq 7.
  \end{equation}
  For $m=7$, we have $7\begin{pmatrix}7-1 \\ \lceil \frac{7-1}{2} \rceil  \end{pmatrix} = 7\cdot 20 =140 > 128 = 2^7$;
  for $m=8$, the inequality evaluates to $280>256=2^8$.
  To prove the claim, we proceed by induction.
  \paragraph{Case 1: ``$m$ even''}
  In this case, $\lceil \frac{m}{2} \rceil = \frac{m}{2} = \lceil \frac{m-1}{2} \rceil$ and
  \begin{subequations}
    \begin{align}
     \nonumber
      (m+1) \begin{pmatrix}m \\ \lceil \frac{m}{2} \rceil  \end{pmatrix} 
        & = (m+1) \left(\begin{pmatrix}m-1 \\ \frac{m}{2}-1  \end{pmatrix} + 
                        \begin{pmatrix}m-1 \\ \frac{m}{2}  \end{pmatrix}\right)\\
      \label{eq:aux1}
        & = 2(m+1) \begin{pmatrix}m-1 \\ \lceil \frac{m-1}{2}\rceil  \end{pmatrix}
          > 2(m+1)\frac{2^m}{m} > 2^{m+1},
    \end{align}
  \end{subequations}
 where we have used the symmetry in the Pascal triangle ($m-1$ is odd) and the induction hypothesis
 to arrive at \eqref{eq:aux1}.
 \paragraph{Case 2: ``$m$ odd''}
  In this case, $\lceil \frac{m}{2} \rceil = \frac{m+1}{2}$ and
  \begin{subequations}
    \begin{align}
     \nonumber
      (m+1) \begin{pmatrix}m \\ \lceil \frac{m}{2} \rceil  \end{pmatrix} 
        & = (m+1) \left(\begin{pmatrix}m-1 \\ \frac{m+1}{2}-1  \end{pmatrix} + 
                        \begin{pmatrix}m-1 \\ \frac{m+1}{2}  \end{pmatrix}\right)\\
      \label{eq:aux2}
        & = (m+1) \left(\begin{pmatrix}m-1 \\ \lceil \frac{m-1}{2} \rceil  \end{pmatrix} + 
                        \begin{pmatrix}m-1 \\ \lceil \frac{m}{2} \rceil  \end{pmatrix}\right).
    \end{align}
  \end{subequations}
 Note that $\begin{pmatrix}m-1 \\ \lceil \frac{m}{2} \rceil  \end{pmatrix}$ is the second-to-largest
 term in the binomial expansion of $(1+1)^{m-1}$.
 Moreover, since $m-1$ is even, the relation to the largest term is 
 \[
  \begin{pmatrix}m-1 \\ \lceil \frac{m}{2} \rceil  \end{pmatrix} 
  = \frac{m-1}{m+1} \begin{pmatrix}m-1 \\ \lceil \frac{m-1}{2} \rceil  \end{pmatrix}.
 \]
 Substituting in \eqref{eq:aux2} and applying the induction hypothesis gives
 \[
  (m+1) \begin{pmatrix}m \\ \lceil \frac{m}{2} \rceil  \end{pmatrix} 
    > (m+1) \left( \frac{2^m}{m} + \frac{m-1}{m+1}\frac{2^m}{m}  \right)
    = \left(\frac{m+1}{m}+\frac{m-1}{m} \right)2^m = 2^{m+1}.
 \]
 Using \eqref{eq:thompson_aux2} rather than \eqref{eq:thompson_aux}
 in Thompson's original proof
 leads to the improved bound of $\|z_k(X;Y)\|\leq \mu^k$ for
 $k\geq 7$.

 We tackle the terms involving words of lengths $k=2,3,\ldots, 6$ manually.
 The reference \cite{VanBruntVisser2015} lists explicit expressions of the 
 summands in the Goldberg BCH series up to $z_8$.
 The first three of them read 
\begin{subequations}
    \begin{align}
    \nonumber z_2(X,Y) & = \frac{1}{2}(XY-YX)
                       \Rightarrow \|z_2(X,Y)\| \leq \frac{2}{2}\mu^2. \checkmark\\
    \nonumber z_3(X,Y) & = \frac{1}{12}\bigl(X^2Y-2XYX+XY^2+YX^2-2YXY + Y^2X\bigr)\\
    \nonumber          & \Rightarrow \|z_3(X,Y)\| \leq \frac{8}{12} \mu^3. \checkmark\\
    \nonumber z_4(X,Y) & = \frac{1}{24}\bigl(X^2Y^2-2XYXY+2YXYX-Y^2X^2\bigr)
                        \Rightarrow \|z_4(X,Y)\| \leq \frac{6}{24} \mu^4. \checkmark
  \end{align}
\end{subequations}
 The expressions for $z_5(X,Y)$ and $z_6(X,Y)$ are too cumbersome to be restated here.
 However, for our purposes, a very rough counting argument is sufficient:
 The expression for $z_5(X,Y)$ features $30$ length-$5$ words with non-zero Goldberg coefficient
 and the largest Goldberg coefficient is $\frac{1}{30}$.
 Hence, $\|z_5(X,Y)\| = \|\sum_{w, |w|=5}{g_w w}\| < \frac{30}{30} \mu^5. \checkmark$
 (A more careful consideration reveals $\|z_5(X,Y)\| \leq \frac{176}{720} \mu^5$.)

 The expression for $z_6(X,Y)$ features $28$ length-$6$ words with non-zero Goldberg coefficient
 and the largest Goldberg coefficient is $\frac{1}{60}$.
 Hence, $\|z_6(X,Y)\| = \|\sum_{w, |w|=6}{g_w w}\| \leq \frac{28}{60} \mu^6. \checkmark$
\end{proof}
\section{Norm bound for the matrix logarithm}
\label{app:logm_bound}
\begin{proposition}
\label{prop:log_exp_bound}
Let $C\in \R^{p\times p}$ be skew-symmetric with $\|C\|_2 < \pi$.
Then \[\|\exp_m(C)-I\|_2<\|C\|_2.\]
\end{proposition}
\begin{proof}
 Since $C$ is skew-symmetric, it features an EVD $C = Q \Lambda Q^H$ with $\Lambda = diag(\lambda_1,\ldots,\lambda_p) = diag(i\varphi_1,\ldots, i\varphi_p)$, where $\varphi \in \left(-\pi, \pi\right)$
 and $\max_j |i\varphi_j| = \|C\|_2$.
 Therefore, $\exp_m(C) =  Q \exp_m(\Lambda) Q^H$ with $\exp_m(\Lambda) = diag(e^{i\varphi_1},\ldots, e^{i\varphi_p})$ and
 \[
   \|\exp_m(C)-I\|_2 = \max_j |e^{i\varphi_j} - 1| < \max_j |\varphi_j| = \|C\|_2.
 \]
 (The latter estimate may also be deduced from Fig. \ref{fig:Plot_logm_norm_estimate}.)
\end{proof}

\begin{proposition}
\label{prop:logm_bound}
Let $V\in O_{n\times n}$ be such that $\|V-I\|_2 < r< 1$.
Then \[\|\log_m(V)\|_2<r \frac{\sqrt{1-\frac{r^2}{4}}}{1-\frac{r^2}{2}}.\]
\end{proposition}
\begin{proof}
 Let $E = V-I$. The matrices $V$ and $E$ share the same (orthonormal) basis of eigenvectors $Q$
 and the spectrum of $V$ is precisely the spectrum of $E$ shifted by $+1$.
 By assumption, $ r > \|E\|_2 = \max_{\mu\in \sigma(E)}{|\mu|}$.
 Hence, the eigenvalues $\lambda\in \sigma(V)$ are complex numbers of modulus one of the form
 $\lambda = e^{i\alpha} = 1 + \mu$, with $|\mu| < r$.
 Thus, $\lambda$ lies on the unit circle but within a ball of radius $r$ around $1\in \C$,
 see Fig. \ref{fig:Plot_logm_norm_estimate}.
 The maximal angle $\alpha$ for such a $\lambda$ is bounded by the slope of the line that starts in $0\in \C$
 and crosses the points of intersection of the two circles $\{|z| < 1\}$ and $\{|z-1| < r\}$.
 The intersection points are $(x_s,\pm y_s) = \left(1-\frac{r^2}{2}, \pm r\sqrt{1-\frac{r^2}{4}}\right)$.
 Therefore 
 \[
   |\alpha| < \arctan\left(\frac{y_s}{x_s}\right) =  \arctan\left(\frac{r\sqrt{1-\frac{r^2}{4}}}{1-\frac{r^2}{2}}\right) <  r\frac{\sqrt{1-\frac{r^2}{4}}}{1-\frac{r^2}{2}}.
 \]
 As a consequence,
 \[
  \|\log_m(V)\|_2 = \|Q\log_m(\Lambda) Q^H\|_2 
    = \max_{\lambda \in \sigma(V)} |\ln(\lambda)| = \max_{\lambda=e^{i\alpha} \in \sigma(V)} |i\alpha| 
    < r \frac{\sqrt{1-\frac{r^2}{4}}}{1-\frac{r^2}{2}}.
 \]
\end{proof}
\begin{figure}[htbp]
  \centering
  \includegraphics[width=0.9\textwidth]{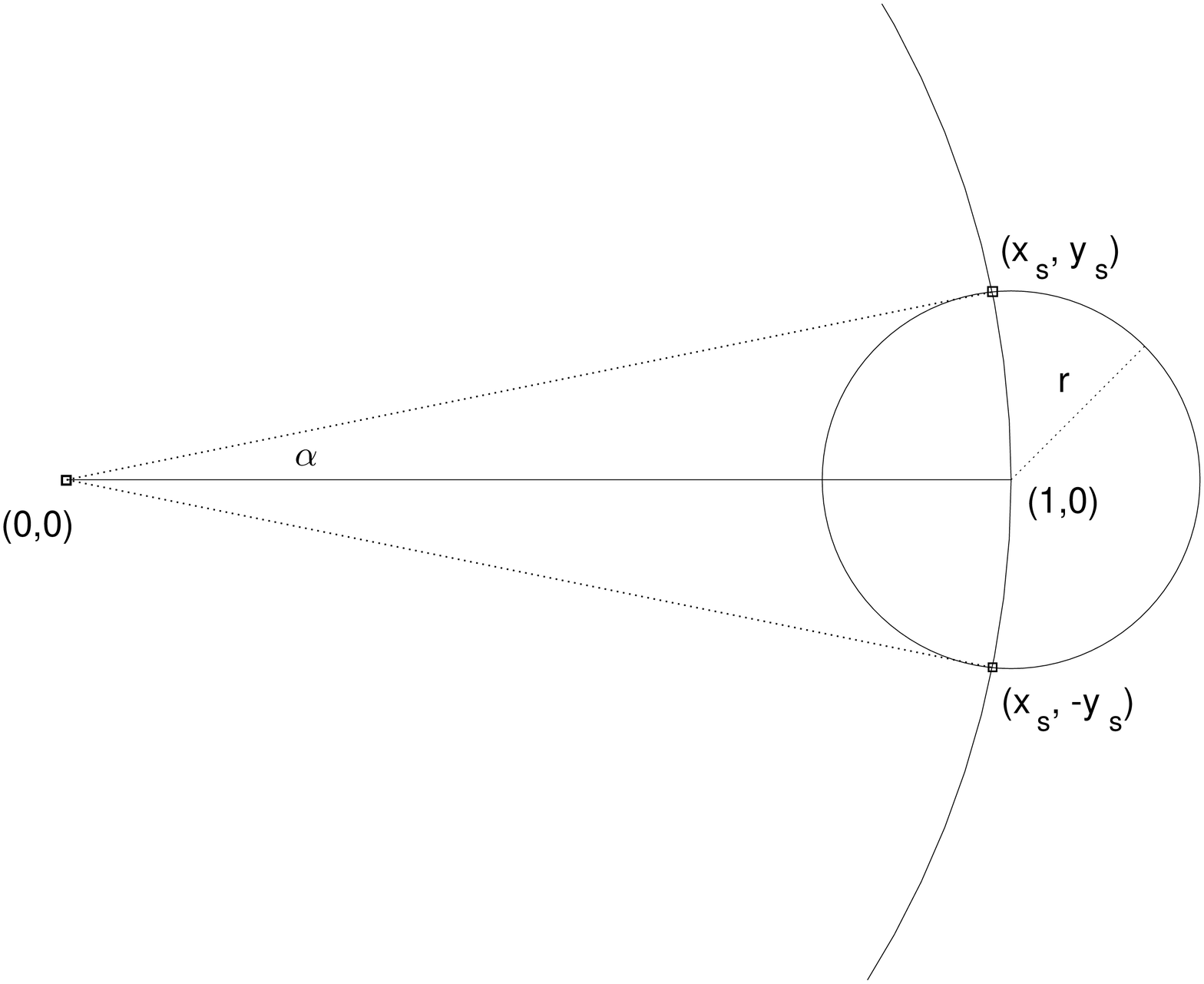}
  \caption{Geometrical illustration of Proposition \ref{prop:logm_bound} in the complex plane.}
    \label{fig:Plot_logm_norm_estimate}
\end{figure}
%
%
\begin{lemma}
 \label{lem:norm_startV0}
  Let $U,\tilde{U} \in St(n,p)$ with $\|U-\tilde{U}\|_2 < \epsilon$. 
  Let $M, N, X_0,Y_0$ and $V_0 := \begin{pmatrix}M&X_0\\N&Y_0\end{pmatrix} \in O_{2p\times 2p}$
  be as constructed in the first steps of Alg. \ref{alg:Stlog}.

  Then
  \begin{equation}
   \label{eq:start_cond_V0}
    \|\log_m(V_0)\|_2 < 2\e \frac{\sqrt{1-\e^2}}{1-2\e^2}.
  \end{equation}
\end{lemma}
\begin{proof}
Because $V_0$ is orthogonal,
\begin{subequations}
  \begin{align}
   \label{eq:num_rad1}
   1 &= \|V_0\|_2 \geq \nu(V_0) = \max_{\|w\|_2=1}{|w^H\begin{pmatrix}M&X_0\\N&Y_0\end{pmatrix} w|}\\
   \label{eq:num_rad2}
     &\geq \max_{\|v\|_2=1}{|(0, v^H)\begin{pmatrix}M&X_0\\N&Y_0\end{pmatrix} \begin{pmatrix}0\\v\end{pmatrix} |} = \|Y_0\|_2,
  \end{align}
\end{subequations}
where $\nu(V_0)$ denotes the numerical radius of $V_0$, see \cite[eq. 1.21, p. 21]{Greenbaum1997}.
Likewise, $\|M\|_2 \leq 1$ so that the singular values of $M$ and $Y_0$ range between $0$ and $1$.
Moreover, by the Procrustes preprocessing outlined at the end of Section \ref{sec:alg},
\[
 \|N\|_2=\|X_0\|_2 < \e, \quad   \|M-I\|_2 <\epsilon, \quad \|Y_0-I_p\|_2 <\e^2,
\]
see \eqref{eq:start_norms}.
%
%
%
Combining these facts, we obtain $V = I+(V-I) = I + E$, where
\begin{subequations}
  \begin{align}
   \nonumber
   \|E\|_2 &= \|\begin{pmatrix}M-I&X_0\\N&Y_0-I\end{pmatrix}\|_2 \leq \|\begin{pmatrix}M-I& 0 \\ 0 & Y_0-I\end{pmatrix} + \begin{pmatrix} 0 & X_0\\N & 0\end{pmatrix} \|_2\\
   \nonumber
     &\leq \max\{\|M-I\|_2, \|Y_0-I\|_2 \} + \max\{\|N\|_2, \|X_0\|_2\} < \max\{\e, \e^2\} +\e = 2\e.
  \end{align}
\end{subequations}
Applying Proposition \ref{prop:logm_bound} to $V = I + E$ proves the claim.
\end{proof}
%
%
%
\section{A critical special case}
\label{supp:critical_convergence}
We present an example that shows that Alg. \ref{alg:Stlog} may fail at computing
$Log^{St}_{U}(\tilde{U})$ even for $U,\tilde{U}\in St(n,p)$ that are only a Riemannian distance
of $dist(U,\tilde{U}) = \frac{\pi}{2}$ apart.

Consider $n=4, p=2$ and set 
\[   
  U = \frac{1}{2}\begin{pmatrix} 1&1&1&1\\ 1&1&-1&-1   \end{pmatrix}^T\in St(4,2), \quad
  \Delta = \frac{1}{2}\begin{pmatrix} -1&1&-1&1\\ 0&0&0&0   \end{pmatrix}^T\in T_{U}St(4,2).
\]
Note that $\Delta^T U = A=0$ and that $\Delta = QR$ with $R=\begin{pmatrix}1&0\\0&0\end{pmatrix}$
is the qr-decomposition of the tangent vector $\Delta$.
Hence, the Stiefel exponential \eqref{alg:Stlog} applied to this data set yields
\[
 \tilde{U}(t) = Exp^{St}_{U}(t\Delta) = (U,Q)\exp_m\left(\begin{pmatrix}0&-tR\\tR&0\end{pmatrix}\right) \begin{pmatrix}I_2\\0\end{pmatrix}.
\]
Because of the simple structure of $R$, the matrix exponential can be computed explicitly
\[
 \exp_m\left(\begin{pmatrix}0&-tR\\tR&0\end{pmatrix}\right) 
 = \begin{pmatrix}\cos(t) &0 & -\sin(t)&0\\
		  0       &1 & 0       &0\\
		  \sin(t) &0 & \cos(t) &0\\
		  0       &0 & 0       &1\end{pmatrix}.
\]
Recall from Section \ref{sec:Stiefel_essentials} that $dist(U,\tilde{U}(t)) = \sqrt{\langle t\Delta,t\Delta\rangle}_U$,
which in this setting evaluates to $t$, since $\Delta$ is of unit norm also with respect to the canonical metric.
For $t=\frac{\pi}{2}$, we obtain 
\[ 
 \tilde{U}:=  \tilde{U}\left(\frac{\pi}{2}\right) = U\begin{pmatrix}0&0\\0&1\end{pmatrix} + Q\begin{pmatrix}1&0\\0&0\end{pmatrix}
    = \frac{1}{2}\begin{pmatrix} -1&1&-1&1\\ 1&1&-1&-1   \end{pmatrix}^T\in St(4,2).
\]
If we now apply Alg. \ref{alg:Stlog} to the matrix pair $U, \tilde{U}$, then we obtain in step 3 of the algorithm
a corresponding
\[
 V_0 = \begin{pmatrix}0 & 0 &  1 & 0\\
		      0 & 1 &  0 & 0\\
		      1 & 0 &  0 & 0\\
		      0 & 0 &  0 & 1\end{pmatrix},
\]
which features $-1$ as an eigenvalue and thus leads to a failure in the principal matrix logarithm.
The problem here is the ambiguity in the orthogonal completion.
If we replace the first row of the above $V_0$ with its negative,
then we have still a valid orthogonal completion, and the method works.
This example suggests that in a practical implementation of Alg. \ref{alg:Stlog}, one should try and
explore strategies to compute a suitable starting iterate $V_0$ with small $\|\log_m(V_0)\|_2$.
\section{Why is the Grassmann case simpler than the Stiefel case?}
\label{supp:Grassmann_case}
An important matrix manifold that is related with the Stiefel manifold and that arises frequently in applications is
the {\em Grassmann manifold}. It is defined as the set of all $p$-dimensional subspaces $\mathcal{U}\subset \R^n$, i.e.,
\[
  Gr(n, p):= \{\mathcal{U}\subset \R^n| \quad \mathcal{U} \mbox{ subspace, dim}(\mathcal{U}) = p\}.
\]
In this supplementary section, I give sketches for derivations for the Riemannian exponential and logarithm
on the Grassmannian. Closed-form expressions for both mappings are known from the literature and I try to explain why the Stiefel case is more difficult.
For background theory, the reader is referred to \cite{AbsilMahonySepulchre2008}, \cite{EdelmanAriasSmith1999}.

The Grassmann manifold can be realized as a quotient manifold of the Stiefel manifold under actions of the orthogonal group via
\begin{equation}
\label{eq:Grassmann_quotient}
  Gr(n, p) = St(n, p)/O_{p\times p} = \{[U]| \quad U\in St(n, p)\}.
\end{equation}
The quotient view point allows for using points $U\in St(n,p)$ as representatives for
points $[U] \in Gr(n,p)$, i.e., subspaces, see \cite{EdelmanAriasSmith1999} for details.
For any matrix representative $U\in St(n, p)$ of $\mathcal{U}=[U]\in Gr(n, p)$,
the tangent space at $\mathcal{U}$ is represented by
\[
 T_{\mathcal{U}}Gr(n, p) = \left\{\Delta \in \R^{n\times p}|\quad U^T\Delta = 0\right\}\subset \R^{n\times p}.
\]
This representation also stems from considering $Gr(n,p)$ as a quotient manifold with $St(n,p)$ as the total space. 
In fact, the tangent space of the Stiefel manifold can be decomposed into the so-called vertical space
and the horizontal space with respect to the quotient mapping, $T_{U}St(n, p) = \mathcal{V}_U \oplus \mathcal{H}_U$,
see \cite[Problem 3.8]{Lee1997riemannian},
\cite[\S3.5.8]{AbsilMahonySepulchre2008}, \cite[\S 2.3.2]{EdelmanAriasSmith1999}.
The explicit representation of vectors in $T_{[U]}Gr(n, p)$ that we have introduced
above corresponds to the identification of the actual abstract
tangent space $T_{[U]}Gr(n, p)$ with the horizontal space $\mathcal{H}_U$.

From the quotient perspective, Grassmann tangent vectors are special Stiefel tangent vectors $\Delta = U_0A+(I-UU^T)T$,
namely those associated with the special skew-symmetric matrix $A=0\in\R^{p\times p}$, cf. \eqref{eq:tang2}.
Hence, we may use the Stiefel exponential to compute the Grassmann exponential:
\begin{itemize}
 \item Given $\Delta\in T_{[U]}Gr(n, p)$, compute the qr-decomposition $(I-U_0U_0^T)\Delta = \Delta = Q_ER_E.$
 \item Compute the matrix exponential
   \begin{equation}
    \label{eq:baby_log_grass}
    \begin{pmatrix}M\\ N_E\end{pmatrix}
          := \exp_m\left(\begin{pmatrix}0 & -R_E^T\\ R_E & 0\end{pmatrix}\right)\begin{pmatrix}I_p\\ 0\end{pmatrix}.
    \end{equation}
  \item Return $\tilde{U} = Exp^{Gr}_{[U_0]}(\Delta) = [U_0M + Q_EN_E]$.
\end{itemize}
It is precisely the extra upper-left zero-block in the matrix exponential in \eqref{eq:baby_log_grass}, that
makes the Grassmann case easier to tackle than the Stiefel case:
By using the SVD $R_E = \Phi \Sigma D^T$ and the series expansion of $\exp_m$, it is straight-forward to show that
   \begin{equation}
    \label{eq:matrix_exp_Grass}
      \exp_m\left(\begin{pmatrix}0 & -R_E^T\\ R_E & 0\end{pmatrix}\right) 
      = \begin{pmatrix}D & 0\\ 0 & \Phi\end{pmatrix}
	\begin{pmatrix}\cos(\Sigma) & -\sin(\Sigma)\\  \sin(\Sigma) & \cos(\Sigma)\end{pmatrix}
	\begin{pmatrix}D^T & 0\\ 0 & \Phi^T\end{pmatrix},
    \end{equation}
which gives
   \begin{equation}
    \label{eq:grass_exp}
    \begin{pmatrix}M\\ N_E\end{pmatrix}
          =  \begin{pmatrix}D\cos(\Sigma)D^T\\\Phi \sin(\Sigma)D^T \end{pmatrix}.
    \end{equation}
(In the above formulae, it is understood that $\sin$ and $\cos$ are to be applied pointwise to the diagonal elements of the diagonal matrix $\Sigma$.)
Eventually, we arrive at
\[
 Exp^{Gr}_{\mathcal{U}_0}(\Delta) = [U_0M+ QN_E] = [U_0D\cos(\Sigma)D^T + Q \Phi\sin(\Sigma)D^T].
\]
Instead of starting with the qr-decomposition $\Delta = Q_ER_E$, we now see that we could have directly worked with the SVD
$\Delta = \hat{Q}\Sigma D^T (= (Q \Phi) \Sigma D^T)$, which yields
$Exp^{Gr}_{\mathcal{U}_0}(\Delta) = [U_0D\cos(\Sigma)D^T + \hat{Q}\sin(\Sigma)D^T].$

This is exactly the expression that Edelman et al. have found in \cite[Thm. 2.3]{EdelmanAriasSmith1999}
for the Riemannian exponential on $Gr(n,p)$ and the derivation above can be considered as a 'thin SVD'-version of \cite[Thm. 2.3, Proof 2, p. 320]{EdelmanAriasSmith1999}.
%
%

The inverse of this mapping, i.e., the Riemannian logarithm on $Gr(n,p)$ can be deduced as follows:
Consider $[U_0],[\tilde{U}]\in Gr(n,p)$. Under the assumption that $[\tilde{U}]$ is sufficiently close to $[U]$,
it holds that $\tilde{U} = U_0M + QN$ and the task is to find $M,N\in\R^{p\times p}$ and $Q\in St(n,p)$ such that 
$Q^TU_0 = 0$.
The first matrix factor $M$ is uniquely determined by $U_0^T\tilde{U} = M$.
We obtain candidates for $Q,N$ by computing the qr-decomposition $(I-U_0U_0^T)\tilde{U} = QN$.
Yet, in order to reverse \eqref{eq:grass_exp}, it is require to work with consistent coordinates.
Taking \eqref{eq:matrix_exp_Grass}, \eqref{eq:baby_log_grass} into account, this is established
by setting $N_L = NM^{-1}$ and computing the SVD $N_L = \Phi S D^T$, because
by defining $\Sigma = \arctan(S)$, we can decompose
\[
 N_L = \Phi S D^T = \Phi \tan(\Sigma) D^T =  \Phi \sin(\Sigma) D^T D (\cos(\Sigma))^{-1} D^T.
\]
This shows that the choice $R_L:= \Phi \Sigma D^T$ yields a tangent vector $\Delta := QR_L$
such that 
\begin{eqnarray}
\nonumber
    Exp^{Gr}_{\mathcal{U}_0}(\Delta) &=& 
    \left[ (U_0, Q)\exp_m\left(\begin{pmatrix}0 & -R_L^T\\ R_L & 0\end{pmatrix}\right)\begin{pmatrix}I_p\\ 0\end{pmatrix}\right]\\
\nonumber
    &=& [U_0 D\cos(\Sigma) D^T + Q\Phi \sin(\Sigma) D^T]  = [\tilde{U}].  
\end{eqnarray}
Note that $QN_L = Q\Phi S D^T \stackrel{\mbox{SVD}}{=} QNM^{-1} = (I-U_0U_0^T)\tilde{U} M^{-1}$.
Hence, we now see that we could have directly started with the SVD of $\hat{Q} S D^T = (I-U_0U_0^T)\tilde{U} M^{-1}$
to arrive at 
\[
 Log^{Gr}_{\mathcal{U}_0}(\tilde{\mathcal{U}}) = \Delta = \hat{Q} \arctan(S) D^T.
\]
This is the well-known closed-form of the Grassmann logarithm.
Unfortunatley, I was not able to track down the original derivation. The earliest appearance in the literature
that I found was \cite[Alg. 3]{BegelforWerman2006}. However, this reference only mentions the above formuala
but does not cite a source. 
In summary, the Grassmann case is easier to deal with because of the extra off-diagonal block structure in
the associated matrix exponential \eqref{eq:baby_log_grass}, which leads to a CS-decomposition in \eqref{eq:matrix_exp_Grass}
by a {\em similarity transformation}; compare this to \cite[Thm. 2.6.3, p.78]{GolubVanLoan}.
\section{MATLAB code}
\subsection{Alg. \ref{alg:Stlog}}
\label{app:code}
\begin{verbatim}
%
function [Delta, k, conv_hist, norm_logV0] = ...
                                 Stiefel_Log_supp(U0, U1, tau)
%-------------------------------------------------------------
%@author: Ralf Zimmermann, IMADA, SDU Odense
%
% Input arguments      
%  U0, U1 : points on St(n,p)
%     tau : convergence threshold
% Output arguments
%   Delta : Log^{St}_U0(U1), 
%           i.e. tangent vector such that Exp^St_U0(Delta) = U1
%       k : iteration count upon convergence
% supplementary output
%  conv_hist : convergence history
% norm_logV0 : norm of matrix log of first iterate V0
%-------------------------------------------------------------
% get dimensions
[n,p] = size(U0);
% store convergence history
conv_hist = [0];

% step 1
M = U0'*U1;
% step 2
[Q,N] = qr(U1 - U0*M,0);   % thin qr of normal component of U1
% step 3
[V, ~] = qr([M;N]);                   % orthogonal completion

% "Procrustes preprocessing"
[D,S,R]      = svd(V(p+1:2*p,p+1:2*p));
V(:,p+1:2*p) = V(:,p+1:2*p)*(R*D');
V            = [[M;N], V(:,p+1:2*p)];  %          |M  X0|
                                       % now, V = |N  Y0| 
% just for the record
norm_logV0 = norm(logm(V),2);
                                                                           
% step 4: FOR-Loop
for k = 1:10000
    % step 5
    [LV, exitflag] = logm(V);
                                  % standard matrix logarithm
                                  %             |Ak  -Bk'|
                                  % now, LV =   |Bk   Ck |
    C = LV(p+1:2*p, p+1:2*p);     % lower (pxp)-diagonal block
    % steps 6 - 8: convergence check
    normC = norm(C, 2);
    conv_hist(k) = normC;
    if normC<tau;
        disp(['Stiefel log converged after ', num2str(k),...
              ' iterations.']);
        break;
    end
    % step 9
    Phi = expm(-C);              % standard matrix exponential
    % step 10
    V(:,p+1:2*p) = V(:,p+1:2*p)*Phi;   % update last p columns
end
% prepare output                         |A  -B'|
% upon convergence, we have  logm(V) =   |B   0 | = LV
%     A = LV(1:p,1:p);     B = LV(p+1:2*p, 1:p)
% Delta = U0*A+Q*B
Delta = U0*LV(1:p,1:p) + Q*LV(p+1:2*p, 1:p);
return;
end
\end{verbatim}
%
\textbf{Note:} The performance of this method may be enhanced
by computing $\exp_m$, $\log_m$ via a Schur decomposition.
\section{MATLAB code corresponding to Section 5.2}
\label{supp:ex52}
\paragraph{First experiment discribed in Section 5.2}
\begin{verbatim}
%-------------------------------------------------------------
% script_Stiefel_Log_supp52.m
% %@author: Ralf Zimmermann, IMADA, SDU Odense
%-------------------------------------------------------------
clear;
% set dimensions
n = 10;
p = 2;
% fix stream of random numbers for reproducability
s = RandStream('mt19937ar','Seed',1);
% set number of random experiments
runs = 100;
dist = 0.4*pi;
average_iters = 0;
for j=1:runs
    %create random stiefel data
    [U0, U1, Delta] = create_random_Stiefel_data(s, n, p, dist);

    % 'project' Delta onto St(n,p) via the Stiefel exponential
    U1 = Stiefel_Exp_supp(U0, Delta);
    % compute the Stiefel logarithm
    [Delta_rec, k] = Stiefel_Log_supp(U0, U1, 1.0e-13);
                      % uncomment the following lines to check 
                         % if Stiefel logarithm recovers Delta
    %norm(Delta_rec - Delta)
    average_iters = average_iters +k;
end
average_iters = average_iters/runs;
disp(['The average iteration count of the Stiefel log is ',...
      num2str(average_iters)]);    
      
% EOF: script_Stiefel_Log_supp52.m
%-------------------------------------------------------------
 \end{verbatim}
%
%
\paragraph{Second experiment discribed in Section 5.2}
\begin{verbatim}
%-------------------------------------------------------------
% script_Stiefel_Log_supp52b.m
% %@author: Ralf Zimmermann, IMADA, SDU Odense
%-------------------------------------------------------------
clear; close all;

dist = 0.44*pi;
%-------------------------------------------------------------
% set dimensions
n = 10;
p = 2;
% fix stream of random numbers for reproducability
s = RandStream('mt19937ar','Seed',1);

%create random stiefel matrix:
[U0, U1, Delta] = create_random_Stiefel_data(s, n, p, dist);
norm_U0_U1 = norm(U0 - U1,2)
% compute the Stiefel logarithm
tic;
[Delta_rec, k, conv_hist1, norm_logV01] = ...
                           Stiefel_Log_supp(U0, U1, 1.0e-13);
toc;
norm_recon11 = norm(Delta_rec - Delta)
%-------------------------------------------------------------

%-------------------------------------------------------------
% reset dimensions
n = 1000;
p = 200;
%create random stiefel matrix:
[U0, U1, Delta] = create_random_Stiefel_data(s, n, p, dist);
norm_U0_U1 = norm(U0 - U1,2)
% compute the Stiefel logarithm
tic;
[Delta_rec, k, conv_hist2, norm_logV02] = ...
                           Stiefel_Log_supp(U0, U1, 1.0e-13);
toc;
norm_recon12 = norm(Delta_rec - Delta)
%-------------------------------------------------------------

%-------------------------------------------------------------
% reset dimensions
n = 1000;
p = 900;
%create random stiefel matrix:
[U0, U1, Delta] = create_random_Stiefel_data(s, n, p, dist);
norm_U0_U1 = norm(U0 - U1,2)
% compute the Stiefel logarithm
tic;
[Delta_rec, k, conv_hist3, norm_logV03] = ...
                           Stiefel_Log_supp(U0, U1, 1.0e-13);
toc;
norm_recon13 = norm(Delta_rec - Delta)
%-------------------------------------------------------------

%-------------------------------------------------------------
% reset dimensions
n = 100000;
p = 500;
%create random stiefel matrix:
[U0, U1, Delta] = create_random_Stiefel_data(s, n, p, dist);
norm_U0_U1 = norm(U0 - U1,2)
% compute the Stiefel logarithm
tic;
[Delta_rec, k, conv_hist4, norm_logV04] = ...
                           Stiefel_Log_supp(U0, U1, 1.0e-13);
toc;
norm_recon14 = norm(Delta_rec - Delta)
%-------------------------------------------------------------


% plot convergence history
figure;
subplot(1,2,1);
semilogy(1:length(conv_hist1), conv_hist1, 'k-s', ...
         1:length(conv_hist2), conv_hist2, 'k:*', ...
         1:length(conv_hist3), conv_hist3, 'k-.o', ...
         1:length(conv_hist4), conv_hist4, 'k--x');
legend('St(10,2)', 'St(1000,200)', 'St(1000,900)',...
       'St(100000,500)')
% EOF: script_Stiefel_Log_supp52b.m
%-------------------------------------------------------------
 \end{verbatim}
\section{MATLAB code corresponding to Section 5.3}
\label{supp:ex53}
\begin{verbatim}
%-------------------------------------------------------------
% script_Stiefel_Log_supp53.m
% @author: Ralf Zimmermann, IMADA, SDU Odense
%-------------------------------------------------------------
clear; close all;

% set dimensions
n = 100;
p = 10;
% fix stream of random numbers for reproducability
s = RandStream('mt19937ar','Seed',1);

%create random stiefel data
[U0, U1, Delta] = create_random_Stiefel_data(s, n, p, 1.0);

% discretize the interval [0.1, 0.9pi] with resolution res
res = 100;
start = 0.01;
t = linspace(start, 0.9*pi, res)'; 
%*************************
% initialize observations
%*************************
% spectral distance U, Uk
norm_U_Uk  = zeros(res,1);
% iterations until convergence
iters_convk = zeros(res,1);
% norm log(V0)
norm_logV0k = zeros(res,1);
% accuracy of the reconstruction
norm_Delta_Delta_rec_k = zeros(res,1);

for k = 1:res
    % 'project' tDelta onto St(n,p) via the Stiefel exponential
    Uk = Stiefel_Exp_supp(U0, t(k)*Delta);
    % compute spectral norm
    norm_U_Uk(k) = norm(U0-Uk,2);

    % execute the Stiefel logarithm
    disp(['Compute log for t=', num2str(t(k))]);
    [Delta_rec, iters_conv, conv_hist, norm_logV0] = ...
        Stiefel_Log_supp(U0, Uk, 1.0e-13);
    
    % store data
    iters_convk(k) = iters_conv;
    norm_logV0k(k) = norm_logV0;
    norm_Delta_Delta_rec_k(k) = norm(t(k)*Delta-Delta_rec, 2);
end

% visualize results
figure;
subplot(1,3,1);
plot(t, iters_convk, 'k-');
legend('iters until convergence');
hold on 
subplot(1,3,2);
plot(t, norm_U_Uk, 'k-');
legend('norm(U_0-U_k)');
hold on 
subplot(1,3,3);
plot(t, norm_logV0k, 'k-');
legend('norm(log_m(V_0))');

figure;
plot(t, norm_Delta_Delta_rec_k);
legend('reconstruction error');
%EOF: script_Stiefel_Log_supp53.m
%-------------------------------------------------------------
 \end{verbatim}
\section{Auxiliary MATLAB functions}
\label{supp:Stexp}
\paragraph{Stiefel exponential}
\begin{verbatim}
%-------------------------------------------------------------
%file: Stiefel_Exp_supp.m
% @author: Ralf Zimmermann, IMADA, SDU Odense
%-------------------------------------------------------------
function [U1] = Stiefel_Exp_supp(U0, Delta)
%-------------------------------------------------------------
% Input arguments      
%   U0    : base point on St(n,p)
%   Delta : tangent vector in T_U0 St(n,p)
% Output arguments
%   U1    : Exp^{St}_U0(Delta), 
%-------------------------------------------------------------
% get dimensions
[n,p] = size(U0);
A = U0'*Delta;                          % horizontal component
K = Delta-U0*A;                             % normal component
[Qe,Re] = qr(K, 0);                   % qr of normal component
% matrix exponential
MNe = expm([[A, -Re'];[Re, zeros(p)]]);
U1 = [U0, Qe]*MNe(:,1:p);
return;
end
%EOF: Stiefel_Exp_supp.m
%-------------------------------------------------------------
 \end{verbatim}

\paragraph{Construction of random data on the Stiefel manifold}
\begin{verbatim}
%-------------------------------------------------------------
%file: create_random_Stiefel_data.m
% @author: Ralf Zimmermann, IMADA, SDU Odense
%-------------------------------------------------------------
function [U0, U1, Delta] =...
    create_random_Stiefel_data(s, n, p, dist)
%-------------------------------------------------------------
% create a random data set 
% U0, U1 on St(n,p),
%  Delta on T_U St(n,p) with canonical norm 'dist',
% which is also the Riemannian distance dist(U0,U1)
%
% input arguments
%     s = random stream (for reproducability)
% (n,p) = dimension of the Stiefel matrices
% dist  = Riemannian distance between the points U0,U1
%         that are to be created
%-------------------------------------------------------------
%create random stiefel matrix:
X = rand(s, n,p);
[U0,~] = qr(X, 0);
% create random tangent vector in T_U0 St(n,p)
A = rand(s, p,p);
A = A-A';            % random p-by-p skew symmetric matrix
T = rand(s, n,p);
Delta = U0*A + T-U0*(U0'*T); 
%normalize Delta w.r.t. the canonical metric
norm_Delta = sqrt(trace(Delta'*Delta) - 0.5*trace(A'*A));
Delta = (dist/norm_Delta)*Delta;
% 'project' Delta onto St(n,p) via the Stiefel exponential
U1 = Stiefel_Exp_supp(U0, Delta);
return;
end
%EOF: create_random_Stiefel_data.m
%-------------------------------------------------------------
 \end{verbatim}
%
%


\begin{thebibliography}{10}

\bibitem{AbsilMahonySepulchre2004}
{\sc P.-A. Absil, R.~Mahony, and R.~Sepulchre}, {\em {R}iemannian geometry of
  {G}rassmann manifolds with a view on algorithmic computation}, Acta
  Applicandae Mathematica, 80 (2004), pp.~199--220,

\bibitem{AbsilMahonySepulchre2008}
{\sc P.-A. Absil, R.~Mahony, and R.~Sepulchre}, {\em Optimization Algorithms on
  Matrix Manifolds}, Princeton University Press, Princeton, New Jersey, 2008,

\bibitem{BolzanoNowakRecht_grouse2010}
{\sc L.~Balzano, R.~Nowak, and B.~Recht}, {\em Online identification and
  tracking of subspaces from highly incomplete information}, in Proceedings of
  Allerton, September 2010.

\bibitem{BegelforWerman2006}
{\sc E.~Begelfor and M.~Werman}, {\em Affine invariance revisited}, 2012 IEEE
  Conference on Computer Vision and Pattern Recognition, 2 (2006),
  pp.~2087--2094,

\bibitem{BennerGugercinWillcox2015}
{\sc P.~Benner, S.~Gugercin, and K.~Willcox}, {\em A survey of projection-based
  model reduction methods for parametric dynamical systems}, SIAM Review, 57
  (2015), pp.~483--531,

\bibitem{EdelmanAriasSmith1999}
{\sc A.~Edelman, T.~A. Arias, and S.~T. Smith}, {\em The geometry of algorithms
  with orthogonality constraints}, SIAM Journal on Matrix Analysis and
  Application, 20 (1999), pp.~303--353,

\bibitem{Gallivan_etal2003}
{\sc K.~Gallivan, A.~Srivastava, X.~Liu, and P.~Van~Dooren}, {\em Efficient
  algorithms for inferences on {G}rassmann manifolds}, in Statistical Signal
  Processing, 2003 IEEE Workshop on, 2003, pp.~315--318,

\bibitem{Goldberg1956}
{\sc K.~Goldberg}, {\em The formal power series for $\log e^xe^y$}, Duke Math.
  J., 23 (1956), pp.~13--21,

\bibitem{GolubVanLoan}
{\sc G.~H. Golub and C.~F. Van~Loan}, {\em Matrix Computations}, The John
  Hopkins University Press, Baltimore -- London, 3~ed., 1996.

\bibitem{Greenbaum1997}
{\sc A.~Greenbaum}, {\em Iterative Methods for solving linear systems}, vol.~17
  of Frontiers in applied mathematics, SIAM Society for Industrial and Applied
  Mathematics, Philadelphia, 1997.

\bibitem{Higham:2008:FM}
{\sc N.~J. Higham}, {\em Functions of Matrices: {Theory} and Computation},
  Society for Industrial and Applied Mathematics, Philadelphia, PA, USA, 2008.

\bibitem{KobayashiNomizu1963}
{\sc S.~Kobayashi and K.~Nomizu}, {\em Foundations of Differential Geometry},
  vol.~I of Interscience {T}racts in {P}ure and {A}pplied {M}athematics no. 15,
  John Wiley \& Sons, New York -- London -- Sidney, 1963.

\bibitem{Lee1997riemannian}
{\sc J.~Lee}, {\em {Riemannian Manifolds: an Introduction to Curvature}},
  Springer Verlag, New York -- Berlin -- Heidelberg, 1997.

\bibitem{Lui2012}
{\sc Y.~Man~Lui}, {\em Advances in matrix manifolds for computer vision}, Image
  and Vision Computing, 30 (2012), pp.~380--388,

\bibitem{MATLAB:2010}
{\sc MATLAB}, {\em version 7.10.0 (R2010a)}, The MathWorks Inc., Natick,
  Massachusetts, 2010.

\bibitem{Thompson1989b}
{\sc M.~Newman, S.~Wasin, and R.~C. Thompson}, {\em Convergence domains for the
  {C}ampbell-{B}aker-{H}ausdorff formula}, {L}inear and {M}ultilinear
  {A}lgebra, 24 (1989), pp.~301--310.

\bibitem{Rahman_etal2005}
{\sc I.~U. Rahman, I.~Drori, V.~C. Stodden, D.~L. Donoho, and P.~Schr\"oder},
  {\em Multiscale representations for manifold-valued data}, SIAM J. Mult.
  Model. Simul., 4 (2005), pp.~1201--1232.

\bibitem{Rentmeesters2013}
{\sc Q.~Rentmeesters}, {\em Algorithms for data fitting on some common
  homogeneous spaces}, PhD thesis, Universit\'{e} {C}atholique de Louvain,
  Louvain, Belgium, July 2013/2015

\bibitem{rossmann2006lie}
{\sc W.~Rossmann}, {\em Lie Groups: An Introduction Through Linear Groups},
  Oxford graduate texts in mathematics, Oxford University Press, 2006,

\bibitem{Thompson1989}
{\sc R.~C. Thompson}, {\em Convergence proof for {G}oldberg's exponential
  series}, Linear Algebra and its Applications, 121 (1989), pp.~3--7.

\bibitem{VanBruntVisser2015}
{\sc A.~Van-Brunt and M.~Visser}, {\em Simplifying the {R}einsch algorithm for
  the {B}aker-{C}ampbell-{H}ausdorff series}.
\newblock arXiv:1501.05034, 2015,

\end{thebibliography}
\end{document}